\documentclass[11pt,a4paper,reqno]{article}
\usepackage{template}
\usepackage{mathrsfs}
\usepackage{graphicx}
\usepackage{ulem}

  \newcounter{kconstant}
  \newcommand{\nck}[1]{\refstepcounter{kconstant}\label{#1}}
  \newcommand{\uck}[1]{k_{\ref{#1}}}
\begin{document}

\title{Random walk on the simple symmetric exclusion process}
\date{\today}
\author{ Marcelo R. Hil\'ario \thanks{Universidade Federal de Minas Gerais, Dep. de Matem\'atica, 31270-901 Belo Horizonte}
  \and
 Daniel Kious\thanks{University of Bath, Department of Mathematical Sciences, Claverton Down, BA2 7AY Bath, UK}
\and
  Augusto Teixeira \thanks{IMPA, Estrada Dona Castorina 110, 22460-320 Rio de Janeiro, RJ - Brazil}}

\maketitle

\begin{abstract}
  We investigate the long-term behavior of a random walker evolving on top of the simple symmetric exclusion process (SSEP) at equilibrium, in dimension one.
  At each jump, the random walker is subject to a drift that depends on whether it is sitting on top of a particle or a hole, so that its asymptotic behavior is expected to depend on the density $\rho \in [0, 1]$ of the underlying SSEP.
  Our first result is a law of large numbers (LLN) for the random walker for all densities $\rho$ except for at most two values $\rho_-, \rho_+ \in [0, 1]$.
  The asymptotic speed we obtain in our LLN is a monotone function of $\rho$.
  Also, $\rho_-$ and $\rho_+$ are characterized as the two points at which the speed may jump to (or from) zero.
  Furthermore, for all the values of densities where the random walk experiences a non-zero speed, we can prove that it satisfies a functional central limit theorem (CLT).
  For the special case in which the density is $1/2$ and the jump distribution on an empty site and on an occupied site are symmetric to each other, we prove a LLN with zero limiting speed.
  Finally, we prove similar LLN and CLT results for a different environment, given by a family of independent simple symmetric random walks in equilibrium.
\end{abstract}

\bigskip

Mathematics Subject Classification (2010): 60K35, 82B43, 60G55

\renewcommand\footnotemark{}
\renewcommand\footnoterule{}
\let\thefootnote\relax\footnotetext{{\bf Keywords:} Random walks, dynamic random environments, renormalization}

\section{Introduction}
\label{s:intro}

Over the last decades the study of the long-term behavior of the position of a particle subject to the influence of a random environment has received great attention from the physics and mathematics community.
In this context, one is usually interested in proving the existence of a well-defined limiting speed for the particle and, once the existence of such a speed is known, to characterize its fluctuations around the average position.

The random environments can be either static (RWRE) or dynamic (RWDRE) depending on whether it is kept fixed or evolves stochastically after the initial configuration is sampled from a given distribution.

For one-dimensional static random environments, since the pioneering work of Solomon \cite{Solomon75}, criteria for recurrence or transience, law of large numbers, central limit theorems, anomalous fluctuation regimes and large deviations have been obtained, see for instance \cite{Solomon75, Kesten75, Sinai82}.
For higher dimensional static environments, important progress has also been achieved, however the knowledge is still modest when compared to the one-dimensional setting (see for instance \cite{Sznitman04,BDR} among many others).
A number of important questions remain open and  there is still much to be understood.
We refer the reader to \cite{MR3059554, Sznitman04} and, more recently \cite{Drewitz14}, for very good reviews on the topic.

The research on random walks on dynamic random environments (RWDRE) was initiated more recently and a number of works have been dedicated to these models, proving laws of large numbers (LLN), central limit theorems (CLT) and deviation bounds.
We provide a short historical background of these works in Section~\ref{ss:related}.

Several of the techniques developed for RWDRE focus on environments with either fast or uniform mixing conditions, see for example \cite{comets2004, RV13, BHT18}.
From a physical perspective, whenever the environment mixes fast one expects the random walk to present diffusive behavior.

Another important class of RWDRE that has received much attention are those that evolve on top of conservative particle systems such as the simple symmetric exclusion process, see \cite{francoiss, BR12, RWonRW} and Section~\ref{ss:related} for a discussion.
To the best of our knowledge, all these works have focused on ballistic and pertubative regimes, as we explain in Section \ref{ss:related}.
In this context, the random walker overtakes the particles of the environment allowing for a renewal structure to be established.
As a consequence the behavior of the random walker in these regimes is also characterized by Gaussian fluctuations and CLT.

However, it is not clear if this diffusive behavior is present for the whole range of parameters of the model.
In fact, in \cite{Avena2012}, simulations indicate that, when the random walk has zero speed, it can display non-diffusive fluctuations, due to the environment's long term memory along the time direction.
In \cite{PhysRevE.97.042116}, diffusivity and trapping effects are predicted both from theoretical physics arguments, as well as simulations.
However, giving a rigorous answer to the asymptotic behavior of this model remains a fascinating open problem in mathematics.

The existence of long-range dependencies, not only brings a set of interesting challenges from the mathematical perspective but, more importantly, also raises the possibility to find non-diffusive behavior and other physically relevant phenomena at the critical zero-speed regime.
This is a major motivation for further investigations, both at the critical value and around its neighborhood.
Let us now describe the setting we consider and present the advances we obtain in this problem.

In this work we consider one-dimensional random walks on top of conservative particle systems starting at equilibrium with density $\rho > 0$.
Although some of the techniques we develop may, in principle, be adapted to other models, we will focus on the case where the environment is either the simple symmetric exclusion process (SSEP) whose law will be denoted $\mathbf{P}^\rho_{EP}$ or a system of independent random walks (PCRW) starting from a Poisson product measure, whose law will be denoted $\mathbf{P}^\rho_{RW}$.
We postpone the mathematical construction of the environments to Section \ref{ss:envi} where we will also precise the meaning of the density $\rho$ in each of the cases.

Once one of these environments is fixed, we define the evolution of the random walk as follows.
Fix two parameters $p_\circ$ and $p_\bullet$ in $[0,1]$ with $p_\circ \leq p_\bullet$.
The random walk starts at the origin and jumps in discrete time.
At the moment of a jump, it inspects the environment exactly in the site where it lies on.
If the site is vacant, the random walk decides to jump to the right with probability $p_{\circ}$ and to the left with probability $1-p_{\circ}$.
If the site is occupied, the random walk decides to jump to the right with probability $p_{\bullet}$ and to the left with probability $1-p_{\bullet}$.
Let us denote  $(X_n)_{n\in \mathbb{N}}$ the trajectory of the random walk.
We are going to revisit the definition of the random walk on Section \ref{ss:randomwalker} where a useful graphical construction is provided.

We will write $\mathbb{P}^{\rho}$ for the joint law of the environment and the random walker when the density of the environment is $\rho$, often called the ``annealed law''.
The details on the construction of this measure are given in Section \ref{s:notation}.

The main contribution of this paper is to develop a technique that allows one to prove LLN and CLT for random walks on a class of dynamical random environments that includes the simple symmetric Exclusion process (SSEP) and the Poisson cloud of independent simple symmetric random walkers (PCRW).

For these two specific models, we use lateral space-time mixing bounds together with a decoupling inequality involving small changes in the density (sprinkling) in order to prove a LLN, i.e.~the existence of an asymptotic speed, for all densities $\rho \in (0,1)$ except at most two values denoted $\rho_-$ and $\rho_+$.
We will provide a characterization of these two possible exceptional densities.
Moreover, we are able to prove a CLT for all densities for which the speed exists and is non-zero.
This is the content of our main result, Theorem \ref{t:speed}.
A simplified version of it is stated below.

\begin{theorem}
  \label{t:intro_speed}
  There exists a deterministic non-decreasing function $v\colon [0,1] \to \mathbb{R}$ and two points $0\leq \rho_- \leq \rho_+ \leq 1$ such that, for every $\rho \in (0,1) \setminus\{\rho_-,\rho_+\}$,
  \begin{equation}
    \frac{X_n}{n} \to v(\rho), \qquad \mathbb{P}^{\rho}-\text{almost surely.}
  \end{equation}
%   where
%  \begin{equation}\label{e:defrho+-}
%    \begin{split}
%      \rho_- & := \sup\{ \rho \in (0, 1): v(\rho) < 0 \}, \\
%      \rho_+ & := \inf\{ \rho \in (0, 1): v(\rho) > 0 \}.
%    \end{split}
%  \end{equation}
  Moreover, for every $\rho \not \in [\rho_-, \rho_+]$, we have that $v(\rho) \neq 0$ and the process
  \begin{equation}
    \left(\frac{X_{\lfloor nt\rfloor}-ntv(\rho)}{\sqrt{n}}\right)_{t\ge0}
  \end{equation}
 converges in law to a non-degenerate Brownian motion.
\end{theorem}

Theorem \ref{t:intro_speed} gives partial answers to Conjectures~3.1 and 3.2 of \cite{Avena2012}, and to open problems stated in \cite{dHKS14,Renato}.
Implicit formulas for $\rho_-$ and $\rho_+$ can be found in Theorem \ref{t:speed}.
It is natural to expect that $\rho_-$ and $\rho_+$ coincide, but it is actually an interesting open problem, as they could a priori be different, which would indicate a transient regime with zero-speed, reminiscent of random walks in (static) random environment.
Also, we are currently unable to determine whether anomalous fluctuations take place for some values of $\rho$  inside the interval $[\rho_-,\rho_+]$.

For the interesting case where the random walk evolves on the simple symmetric  exclusion process with density $\rho=1/2$ and if $p_\bullet=1-p_\circ$, then the speed exists and is equal to $0$, as stated in Theorem~\ref{l:exclusion_symmetric}.
As far as we are aware, the existence of this zero-speed regime was still an open question.

\subsection{Related work}
\label{ss:related}

The first works dedicated to the study of RWDRE focused on the case were  the underlying medium exhibit fast mixing conditions.
A broad range of such conditions have been considered such as: time independence \cite{MR1477651, MR1766340}, strong mixing conditions \cite{comets2004, MR2786643, MR2191225}, exponential mixing rate \cite{denhollander2014, Bethuelsen2018, mountford2015} \cite{MR2123642}, \cite{Avena2017} and fast decay of covariances \cite{BHT18}.
In all the above circumstances, one expects the random walk to exhibit Gaussian fluctuations and to satisfy a functional central limit theorem.

There have been developments on random walks on top of conservative particle systems, such as SSEP \cite{MR3108811, francoiss} or PCRW \cite{BR12, dHKS14, RWonRW, BHSST1, BHSST2}.
Most of the results therein hold in regimes that are perturbative in some parameter: the density of the environment \cite{RWonRW, BHT18, BHSST1, BHSST2}, the rate of evolution of the environment \cite{francoiss} or the local drift experimented by the random walk \cite{Avena2013} are taken to extreme values.
The main idea is, knowing that the random walk would be ballistic in the limiting case, try to prove that it is still ballistic as the parameters of the model approach the limiting values.
From the ballistic behavior, usually LLN and CLT are obtained with renewal techniques.

In \cite{Avena2012}, a (continuous-time) random walk on the SSEP was studied by means of simulations.
There the authors investigated the limiting behavior as a function of three parameters: the density $\rho$ of the SSEP, the rate $\gamma$ of the SSEP and the local drift $p_\bullet$ of the random walk on occupied sites.
They restrict themselves to the case where the local drift on vacant sites satisfy  $p_\circ = 1-p_\bullet$.

Based on their data, they conjecture that LLN should hold for every possible value of $\rho$.
They also conjecture that it is possible to tune the parameters in order to produce regimes in which the fluctuations of the walker around its limiting speed scale super or sub-diffusively.
This phenomena, should be regarded as a manifestation of the strong space-time correlations of the environment which allow, for instance, the existence of traps that survive enough time for being relevant in the long-term behavior.
They also leave as an interesting open question, to determine wether there are some regimes where the walk can be transient with zero speed, which would be reminiscent of similar phenomena that take place for random walks in static random environments.

To the best of our knowledge, the conjectures and open questions presented in \cite{Avena2012} concerning the behavior of the random walker at or near the zero-speed regime have remained largely open and deserve to be further investigated. We also refer the reader to \cite{PhysRevE.97.042116} for a non-rigorous investigation of the possible trapping mechanisms.

\subsection{Overview of the proof}

In Section \ref{s:strategy}, we provide a sequence of statements that build the main steps towards the proofs of our mains results, Theorem \ref{t:speed} and Theorem \ref{l:exclusion_symmetric}. Here, we  give a rougher description of the overall strategy and of the tools we use.

The first step is to give a graphical construction of the process, in Section \ref{s:coupled_rw},  which will be very useful to emphasize some monotonicity properties. From the graphical construction, we obtain a two dimensional space-time picture on which we define a collection of coupled random walks started from each point of space and time. We observe events in {\it boxes}, which are simply finite regions in this space-time landscape.

Most of our proofs are based on renormalization schemes: the idea is to observe some events on larger and larger boxes and prove that, if some bad event happens at some scale then similar events happen in two different boxes at a smaller scale. If these two boxes were independent, then one could obtain (given an initial estimate) that the probability of the bad event decays exponentially fast in the size of the box. Nevertheless, in our case, the boxes are not independent as the dynamical nature of the environment creates  time and space dependencies. The key observation is the following: as the particles of the environment move diffusively, if the boxes are at a space distance at least $D$ and at a time distance not larger than $D^2$, then we should be able to prove that these boxes are {\it essentially independent}, see Proposition \ref{p:decouple}. This is what we call here the {\it lateral decoupling}, refering to the relative positions of the boxes on the space-time landscape.

There are two quantities that are important in our analysis: $v_+(\rho)$ and $v_-(\rho)$, defined in \eqref{e:v_+} and \eqref{e:v_-}. These {\it upper} and {\it lower speeds} are deterministic and well-defined for every value of $\rho$.
Roughly speaking the probability to move at speeds larger than $v_+(\rho)$ should go to zero along a subsequence (and similarly for the probability to go slower than $v_-(\rho)$).

Then we need to prove a few facts in order to be able to conclude the existence of the speed:
\begin{enumerate}
\item[(1)] the probability to go faster than $v_+(\rho)$ or slower than $v_-(\rho)$ over a time $t$ actually decreases fast enough;
\item[(2)] $v_+(\rho)=v_-(\rho)$, which is then our candidate $v(\rho)$ for the speed.
\end{enumerate}
These two points would indeed imply the existence of the speed. Nevertheless, due to the nature of the lateral decoupling, we are only able to decouple nicely events on space-time boxes that are well separated in space. For instance, two boxes with the same space location but different time locations will not decouple nicely. For this reason,  in Lemma \ref{l:deviation_interp} and Theorem \ref{l:v_+=v_-}, we are only able to prove the following:
\begin{enumerate}
\item[(1')] the probability to go faster than $\max(v_+(\rho),0)$ or slower than $\min(v_-(\rho),0)$ over a time $t$ actually decreases fast enough;
\item[(2)] $v_+(\rho)=v_-(\rho)=v(\rho)$.
\end{enumerate}
This is proved using renormalization and the lateral decoupling, and this is not quite enough to obtain the existence of the speed. 
However, note that this already implies the existence of the speed if $v(\rho)=0$.

We still need to prove that if, for instance, $v(\rho)=v_+(\rho)=v_-(\rho)>0$, then the probability to go slower than $v(\rho)$ decays sufficiently fast.
Note that this is not guaranteed by (1') alone.

For this purpose, we first consider a density $\rho_+$ for which we know that $v(\rho_++\varepsilon)>0$ and we prove that, for any environment with density $\rho>\rho_++\varepsilon$, the probability to go slower than $v(\rho_++\varepsilon)$ decays fast, see Proposition \ref{p:ballisticity}.
Therefore, this provides a linear lower bound for the displacement of the random walker, that is, we conclude that the random walker moves ballistically.

Once we have obtained this ballisticity, we conclude in Section \ref{billburr} the existence of the speed (and the CLT) by using regeneration structures developed in  \cite{francoiss} for the SSEP and in \cite{RWonRW} for the PCRW.\\

We want to emphasize that the technique we develop here is not sensitive to the particularities of the environment. Indeed, the only things we actually require from the environment is some monotonicity (roughly speaking, increasing the environment is equivalent to adding some particles) and to satisfy the lateral decoupling, i.e.~to have correlations that are sub-linear in time. Also, we could choose the random walk to be discrete time or continuous-time without too much troubles. Nevertheless, the last block in our proof is the existence of regeneration times and, there, we can only work with the particular environments used in \cite{francoiss}  and \cite{RWonRW}. We believe that one could develop a renewal structure for any relevant environment but, as we believe that all these processes are physically equivalent, we chose to simply use the results that already exist.

\section{Mathematical setting and main results}
\label{s:notation}

\subsection{Environments}
\label{ss:envi}

In this section, we will give the mathematical construction of two dynamic random environments that we consider: the simple symmetric exclusion process (SSEP) and the Poisson cloud of simple random walks (PCRW).
The starting configuration will be distributed so that the environment is in equilibrium, that is, the environmental process is stationary in time.

Our particle systems will start in equilibrium with a distribution parametrized by a value $\rho\in(0,1)$.
As it will become clear below, if $\lambda$ denotes the expected number of particles per site then we will have $\rho = \lambda$ for the SSEP, while  we will have $\rho = 1 - e^{-\lambda}$ for the PCRW.

For a fixed $\rho$, we will denote $\mathbf{P}^{\rho}_{EP}$ and $\mathbf{P}^{\rho}_{RW}$ the law of a SSEP and a PCRW with density parameter $\rho$, respectively.
We also write  $\mathbf{E}^\rho_{EP}$, $\mathbf{E}^\rho_{RW}$ for the corresponding expectations.

We may drop the superscript $\rho$ and/or the subscripts $EP$ and $RW$ when there is no risk of confusion.

\subsubsection{Definition of the SSEP}\label{const_ssep}
The SSEP with density $\rho\in(0,1)$ is a stochastic process
\begin{equation}
\eta^\rho = (\eta^\rho_t(x); x\in \mathbb{Z})_{t\in \mathbb{R}_+}
\end{equation}
whose graphical construction we outline below.

But let us first give an informal introduction to the process.
At time $t=0$ decide whether a site has a particle independently by tossing a coin with success probability $\rho$.
Then, each particle tries to the jump at a rate $\gamma>0$ with equal probability to the right and to the left.
The jump will only be performed in case the landing position does not contain a particle.

Mathematically, we start by fixing $\left(\eta^\rho_0(x),x\in\mathbb{Z}\right)$, a collection of i.i.d.~Bernoulli random variables with mean $\rho$, that is, $\eta^\rho_0(x)$ is equal to $1$ with probability $\rho$ and to $0$ with probability $1-\rho$, independently over $x \in \mathbb{Z}$.
This represents the initial configuration for the SSEP with density $\rho$.

In order to define the evolution of this process, to each unit edge $\{x,x+1\}$ of $\mathbb{Z}$, we associate a real-valued Poisson point process $(T^x_i)_{i\ge0}$ with rate $\gamma$, independently over $x\in\mathbb{Z}$.
The SSEP is defined as follows.

If, for some $t>0$, $x\in\mathbb{Z}$ and $i\ge0$, we have $T^x_i=t$, then
\begin{equation}
\eta^\rho_t(x+k)=\lim_{\substack{t'\to t,\\ t'<\, t}}\eta^\rho_{t'}(x+1-k)\text{ for }k\in\{0,1\}.
\end{equation}
In words, at each arrival of the Poisson point process $(T^x_i)_{i\geq0}$, the sites $x$ and $x+1$ exchange their occupation.
The construction implies that $\eta^\rho_t(x) \in \{0,1\}$ for every $x \in \mathbb{Z}$ and $t \geq 0$.
When $\eta^\rho_t(x) = 1$ we say that there is a particle on site $x$ at time $t$.
Differently, when $\eta^\rho_t(x)=0$ we say that there is a hole on site $x$ at time $t$.

We will denote $\mathbf{P}^{\rho}_{EP}$ the law of $\eta^{\rho}$ as an element of $\mathcal{D}(\mathbb{R}_+, \{0,1\}^{\mathbb{Z}})$, the standard space of c\`adl\`ag trajectories in $\{0,1\}^{\mathbb{Z}}$.
It is a classical fact that the Bernoulli distribution with parameter $\rho$ is an invariant measure for this process.
Hence, under $\mathbf{P}^{\rho}_{EP}$, and for each fixed time $t\ge0$,  $(\eta^\rho_{t}(x))_{x\in\mathbb{Z}}$ is a collection of Bernoulli random variables with parameter $\rho$.
We also write  $\mathbf{E}^\rho_{EP}$ for the corresponding expectations.

Note that the collections of i.i.d.\ random variables $\left(\eta^\rho_0(x),x\in\mathbb{Z}\right)$ for different parameters $\rho \in (0,1)$ can be coupled in such a way that $\eta^\rho_0(x) \geq \eta_0^{\rho'}(x)$ whenever $\rho' < \rho$.
The graphical construction presented above preserves this property for every $t >0$.\\

Finally, note that $\eta^\rho$ depends also on $\gamma$, but. as we fix this parameter throughout the paper. we can safely make this dependency implicit.

\subsubsection{Definition of PCRW}
The PCRW with density parameter $\rho \in (0,1)$ is a stochastic process
\begin{equation}
\eta^\rho = (\eta^\rho_t(x);\, x \in \mathbb{Z})_{t\in\mathbb{R}_+}
\end{equation}
defined in terms of a collection of independent random walks on $\mathbb{Z}$ as we show below.

Fix $\rho \in (0,1)$ and let  $\lambda=-\ln(1-\rho)\in\mathbb{R}_+$.
Now let $\left(\eta^\rho_0(x),x\in\mathbb{Z}\right)$ be an i.i.d.~collection of Poisson random variables of parameter $\lambda$.
Independently, for every $x\in \mathbb{Z}$, we let $(Z_t^{x,i},t\ge0)_{1\le i\le \eta^\rho_0(x)}$ be a collection of lazy, discrete-time, simple random walks started at $x$ that evolve independently by jumping at each time unit by $-1,0,1$ with probabilities  $(1-q)/2,q,(1-q)/2$, respectively, for some $q\in(0,1)$.
Note that, for a given $x$, the collection is empty on the event that $\eta^\rho_t(x)=0$.
Moreover, note that the random walks are indexed by a continuous time parameter $t \geq 0$, so that their trajectories are in $\mathcal{D}(\mathbb{R}_+, \mathbb{Z})$, despite the fact that the jumps can only occur in integer-valued instants of time.

We define the number of walkers at a given time $t \geq 0$ and position $y \in \mathbb{Z}$ as:
\begin{equation}
\eta^\rho_t(y)=\big|\big\{(x,i) : \ x\in\mathbb{Z}, 1\le i\le \eta^\rho_0(x), Z_t^{x,i}=y\big\}\big|.
\end{equation}
Note that $\eta^\rho_t(x) \in \mathbb{N}$.
When $\eta^\rho_t(x) = j$ for some integer $j \geq 1$ we say that there are $j$ particles on site $x$ at time $t$.
Differently, when $\eta^\rho_t(x) = 0$ we say that there is no particle on site $x$ at time $t$.

We will denote $\mathbf{P}^{\rho}_{RW}$ the law of $\eta^\rho$ defined on an appropriate probability space.
We also write $\mathbf{E}^\rho_{RW}$ for the corresponding expectations.
It is well-known that, for any $\lambda>0$, the Poisson product distribution with parameter $\lambda$ is an invariant measure for this process.
Hence, under $\mathbf{P}^{\rho}_{RW}$, and for each fixed $t\ge0$,  $(\eta^\rho_{t}(x))_{x\in\mathbb{Z}}$ is a collection of independent Poisson random variables with parameter $\lambda$.

Note that it is possible to couple the collections of i.i.d.\ Poisson random variables $\left(\eta^\rho_0(x),x\in\mathbb{Z}\right)$ for different parameters $\rho \in (0,1)$ in such a way that $\eta^\rho_0(x) \geq \eta_0^{\rho'}(x)$ for every $x\in\mathbb{Z}$ whenever $\rho' < \rho$.
The dynamics defined above preserves this property for every $t>0$.

\begin{remark}
  One should note that the environment we define here may not be the most natural.
   Indeed, it would be a priori simpler to consider that particles perform independent continuous-time simple random walks, instead of considering discrete-time lazy random walks.
  This choice was done in order to use previously proved results that have been established for the discrete-time case only.
  In fact, all the proofs we present here would work almost verbatim in the continuous case, at the exception of the final step of the paper where we use the regeneration structure and the results of \cite{RWonRW}.
  Nevertheless, as mentioned in \cite{RWonRW}, we believe that it is possible to adapt these results to the continuous case (see \cite{BR12} for similar statements in different context).
\end{remark}

\subsection{The random walker}\label{ss:randomwalker}
In this section, we define a discrete-time random walker $X$ evolving on the SSEP or on the PCRW.

We fix two transition probabilities $p_\bullet,p_\circ\in(0,1)$.
Let $\eta^\rho$ be distributed under either $\mathbf{P}^\rho_{EP}$ or $\mathbf{P}^\rho_{RW}$.
Conditioned on $\eta^\rho=\eta$, we define $(X_n)_{n\ge0}$ such that $X_0=0$ a.s.\ and, for $n\ge0$,
\begin{itemize}
\item if $\eta_n(x) > 0$, then $X_{n+1}=X_n +1$ with probability $p_{\bullet}$  and $X_{n+1}=X_n -1$ with probability $1 - p_{\bullet}$;
\item if $\eta_n(x) = 0$, then $X_{n+1}=X_n +1$ with probability $p_{\circ}$  and $X_{n+1}=X_n -1$ with probability $1 - p_\circ$.
\end{itemize}
We will denote $P^\eta_{p_\bullet,p_\circ}$, or simply $P^\eta$, the {\it quenched law} of $X$, i.e.~the law of $X$ on a fixed environment $\eta^\rho=\eta$. 
We will denote $\mathbb{P}^\rho_{p_\bullet,p_\circ}$, or simply $\mathbb{P}^\rho$, the {\it annealed law} of the walk, in other words, the semi-product $\mathbf{P}_{EP}^\rho\times P^\eta$ or $\mathbf{P}_{RW}^\rho\times P^\eta$.

\begin{remark}
In Section \ref{s:coupled_rw}, we will give an alternative definition of $X$ through a graphical construction which will couple realizations with different starting positions and will be very useful in the course of the proofs.
\end{remark}

Note that either $p_\bullet\ge p_\circ$ or $1-p_\bullet\ge 1-p_\circ$ must hold, hence (by flipping the integer line if necessary), we may assume without loss of generality throughout the paper
\begin{equation}\label{e:probassump}
p_\bullet\ge p_\circ.
\end{equation}
Although there is no loss of generality in imposing this assumption, we will make some statements that rely on it.
In case \eqref{e:probassump} does not hold, then the symmetric statements would hold.

\subsection{Main theorems}
\label{s:main}

In this section we provide the precise statements of our main results.

\begin{theorem}\label{t:speed}
Consider the environment $\eta^\rho$ with density $\rho\in(0,1)$ distributed under either the measure $\mathbf{P}^\rho_{EP}$ or $\mathbf{P}^\rho_{RW}$ and assume that $\eqref{e:probassump}$ holds.
There exists a deterministic non-decreasing function $v\colon (0,1) \to \mathbb{R}$ such that
\begin{equation}\label{mainspeed}
\frac{X_n}{n} \to v(\rho), \qquad \mathbb{P}^{\rho}-\text{almost surely,}
\end{equation}
for every $\rho\in(0,1)\setminus\{\rho_-,\rho_+\}$,  where
\begin{equation}\label{e:defrho+-}
  \begin{split}
    \rho_- & := \sup\{ \rho \in (0, 1): v(\rho) < 0 \}, \\
    \rho_+ & := \inf\{ \rho \in (0, 1): v(\rho) > 0 \}.
  \end{split}
\end{equation}
If $v(\rho_-)=0$, resp.~$v(\rho_+)=0$, then \eqref{mainspeed} holds for $\rho=\rho_-$, resp.~for $\rho=\rho_+$.\\
Moreover, for every $\rho \not \in [\rho_-, \rho_+]$, we have a functional central limit theorem for $X_n$ under $\mathbb{P}^\rho$, that is
\begin{equation}
\left(\frac{X_{\lfloor nt\rfloor}-ntv(\rho)}{\sqrt{n}}\right)_{t\ge0} \stackrel{(d)}{\to} \left(B_t\right)_{t\ge0},
\end{equation}
where $\left(B_t\right)_{t\ge0}$ is a non-degenerate Brownian motion and where the convergence in law holds in the Skorohod topology.
\end{theorem}
In \eqref{e:defrho+-} we use the convention that $\inf{\emptyset} = 1$ and $\sup{\emptyset}=0$.

\begin{remark}
We believe that the speed $v(\cdot)$ is continuous on $(0,\rho_-)\cup(\rho_+,1)$.
Nevertheless, to prove so, one should adapt the definitions of regeneration times introduced in \cite{RWonRW} and in \cite{francoiss}. 
In Section \ref{billburr}, we choose to use the renewal structure from \cite{RWonRW} and \cite{francoiss} as they are. Although we believe that this does not constitute a big conceptual step, it seems quite technical to implement these modifications.
\end{remark}

Note that, in the previous theorem, we do not claim that $\rho_- \neq \rho_+$ and neither that $\rho_-$ and $\rho_+$ are necessarily discontinuity points of $v$.
In case $v$ happens to be continuous at these points (which we strongly expect to be true), then we have a law of large numbers for every density $\rho$.
It is indeed an interesting question whether there are examples of environments for which $v$ is discontinuous. Indeed, the Mott variable-range hopping is studied in \cite{nina}, where Example 2, p.~7, provides an environment and a walk such that the speed, as a function of a certain parameter, has a jump from $0$ to some positive value. Let us roughly describe it: the environment is a point process on the real line where points are randomly spaced, according to some density $1/\gamma$.
On this environment, one can define a random walk which, on the one hand, is more likely to jump on points that are closer to it and, on the other hand, has a bias (denoted $\lambda$ in \cite{nina}) to the right. In \cite{nina}, it is proved that this walk is transient to the right with a speed that jumps from $0$ to a positive value at density $1/\gamma=1/(2-\lambda)$.
Discontinuities of the speed of the walk with respect to the density have also been observed in \cite{BHSST2}.

We are currently unable to study the fluctuations of the random walker for densities in $[\rho_-, \rho_+]$.
It is actually expected that $\rho_- = \rho_+$, but the order of fluctuations is still a controversial issue in the literature, see \cite{Avena2012}, \cite{PhysRevE.97.042116}.
It is a very interesting open problem to determine whether the random walk is diffusive or not in this regime.

We now move to the symmetric case on the SSEP.

\begin{theorem}
\label{l:exclusion_symmetric}
Consider the random walk $(X_n)_n$ on the simple symmetric exclusion process.
If $p_\bullet = 1-p_\circ$ and $\rho=1/2$, then $X_n/n\to 0$, $\mathbb{P}^\rho$-almost surely.
\end{theorem}

Note that by symmetry arguments, if the speed exists then it must be zero.
However, proving Theorem~\ref{l:exclusion_symmetric} is not a trivial task.
Intuitively, one may think that, in order to prove the result, it is necessary to control the trajectory of a walk which comes back very often to its starting position and thus interact with the same particles of the environment a large number of times.
These interactions would create long-range time-dependencies in the trajectory and complicate the analysis.
Our proof nicely gets around this issue and we rather show that trajectories that go far away at positive speed are very unlikely.

\begin{remark}
In \cite{Renato}, the author proves a linear lower bound for the random walk on the exclusion process which is strictly larger than $2p_{\circ}-1$ provided that $p_\bullet>p_\circ$. Thus, one could almost conclude that, for the random walk on the exclusion process with $p_{\bullet}>1/2$ and $p_\circ=1/2$, the law of large numbers (and the CLT) holds with a positive speed for any density $\rho>0$. The reason why we cannot easily state this, is that  \cite{Renato} deals with a continuous time random walk whereas we work with a discrete time random walk. Nevertheless, the proof of \cite{Renato} may be adapted to our case. Even though this seems to be a natural result, there is, to the best of our knowledge, no simple proof of this fact (or even of transience), and the current proof relies on an elaborate multiscale analysis developped by Kesten and Sidoravicius \cite{KS03}.
\end{remark}

\section{Graphical construction and backbone of the proof}
\label{s:strategy}

The first goal of this section is to give a graphical construction of a family of coupled random walks, which we will use extensively throughout the paper. The second goal of the section is to provide the main statements that lead to the proof of our main results, Theorem \ref{t:speed} and Theorem~\ref{l:exclusion_symmetric}.

We will denote $c_0, c_1, \ldots$ and $k_0, k_1, \ldots$  positive numbers whose values are fixed at their first appearance in the text.
These constants may depend on the law of the environment and of the walk.
When a constant depends on other parameters, we shall indicate this at its first appearance, for instance, $c_i(v,\rho)$, is a constant whose value depends on $v$ and $\rho$.
For later appearances, we may omit some of the dependencies and simply write $c_i$ or $c_i(v)$, for example.
Let us define the following notation: for any $w=(x,t)\in\mathbb{R}^2$, we let
\begin{equation}
\pi_1(w)=x\text{ and }\pi_2(w)=t.
\end{equation}
For two real numbers $s,t\in\mathbb{R}$, we denote $s\wedge t$ and $s\vee t$ respectively the minimum and the maximum of $s$ and $t$.
For the rest of the paper, unless precised otherwise, we assume that the law of the environment is either $\mathbf{P}^\rho_{EP}$ or $\mathbf{P}^\rho_{RW}$, with $\rho\in(0,1)$, and we fix $0<p_\circ\le p_\bullet<1$  (thus \eqref{e:probassump} holds), without loss of generality.
\begin{remark}
In our proofs, we could allow for $p_\circ=0$ or $p_\bullet=1$, but we need to rule this out in order to use  the renewal structure from \cite{RWonRW} and \cite{francoiss}, where the authors need this assumption.
\end{remark}

\subsection{Coupled continuous space-time random walks}\label{s:coupled_rw}

Here we use a graphical construction in order to define a family of coupled continuous space-time random walks $(X^w_t,w\in{\mathbb{R}^2},t\ge0)$.
We now informally state the properties of this coupling that will be useful later on.

Each random walk $X^w :=(X^w_t)_{t \geq 0}$ will be such that, $X^w_0=\pi_1(w)$ almost surely.
Furthermore, $X^{w'}$ and $X^w$ coalesce if they ever intersect, that is if $X^{w'}_{t'}=X^w_t$  for some $w,w'\in{\mathbb{R}^2}$ and $t,t'\ge0$, then $X^{w'}_{t'+s}=X^w_{t+s}$ for all $s\ge0$.
Moreover, $(X^{(0,0)}_n)_{n\in\mathbb{N}}$ has the same law as our random walker $(X_n)_{n\in \mathbb{N}}$.

Fix a value $\rho \in (0,1)$ and a realization of the environment $\eta^\rho$. Note that, as $(X_n)_{n\in\mathbb{N}}$ is assumed to be a nearest-neighbor random walk, $X_{2n}\in 2\mathbb{Z}$ and $X_{2n+1}\in \left(1+2\mathbb{Z}\right)$, for every $n\ge0$.
Define the discrete lattice
\begin{equation}
\label{e:def_space_time}
\mathbb{L}_d:=\left( 2\mathbb{Z} \right)^2\cup  \left( (1,1)+\big(2\mathbb{Z}\right)^2 \big),
\end{equation}
where the sum in the RHS stands for the the shift of the $2\mathbb{Z}$ lattice by the vector $(1,1)$.
We will  define the random walks $X^w$ first on this lattice before interpolating them to the whole plane.
For that, we let $\left(U_w\right)_{w\in\mathbb{L}_d}$ be a collection of i.i.d.~uniform random variables on $[0,1]$.
For any  $w=(x,n)\in \mathbb{L}_d$, we set $X^w_0=x$ and define $X^w_1$ in the following manner:
\begin{equation}
X^w_1 =
\begin{cases}
x+2\, \mathbf{1}_{\{U_w\le p_\bullet\}}-1, \text{ if $\eta^\rho_n(x)>0$};\\
x+2\, \mathbf{1}_{\{U_w\le p_\circ\}}-1, \text{ if $\eta^\rho_n(x)=0$}.
\end{cases}
\end{equation}
%\begin{itemize}
%\item if $\eta^\rho_n(x)>0$, then $X^w_1=x+2\mathbf{1}_{\{U_w\le p_\bullet\}}-1$;
%\item if $\eta^\rho_n(x)=0$, then $X^w_1=x+2\mathbf{1}_{\{U_w\le p_\circ\}}-1$.
%\end{itemize}
For any integer $m\ge1$ and we define by induction \begin{equation}
X^w_m=X^w_{m-1}+X^{(X^w_{m-1},m-1)}_1.
\end{equation}
This defines the coupled family $(X^w_n,w\in\mathbb{L}_d,n\in\mathbb{N})$.
Note that $(X^w_n,\pi_2(w)+n)_{n\in \mathbb{N}}$ evolves on $\mathbb{L}_d$.

Having defined the random walker on $\mathbb{L}_d$, we extend its definition to all possible starting points in $\mathbb{R}^2$.
But first we extend it to a continuous version of $\mathbb{L}_d$, defined as follows
\begin{equation}
\mathbb{L}=\{w+ t(1,1), w\in\mathbb{L}_d, 0\le t<1\}\cup \{w+ t(-1,1), w\in\mathbb{L}_d, 0\le t<1\},
\end{equation}
see Figure~\ref{f:L_d}.

For $t\in\mathbb{R}^+$ and $w\in\mathbb{L}_d$, we define
\begin{equation}\label{badcompany}
X^w_t=X^w_{\lfloor t\rfloor}+\left(t-\lfloor t\rfloor\right)\big(X^w_{\lfloor t\rfloor+1}-X^w_{\lfloor t\rfloor}\big).
\end{equation}
This defines the coupled family $(X^w_t,w\in\mathbb{L}_d,t\in\mathbb{R}_+)$.
Note that $(X^w_t,\pi_2(w)+t)_{t\in \mathbb{R}_+}$, evolves on $\mathbb{L}$.

\begin{figure}[h]
  \center
  \includegraphics[scale=.5]{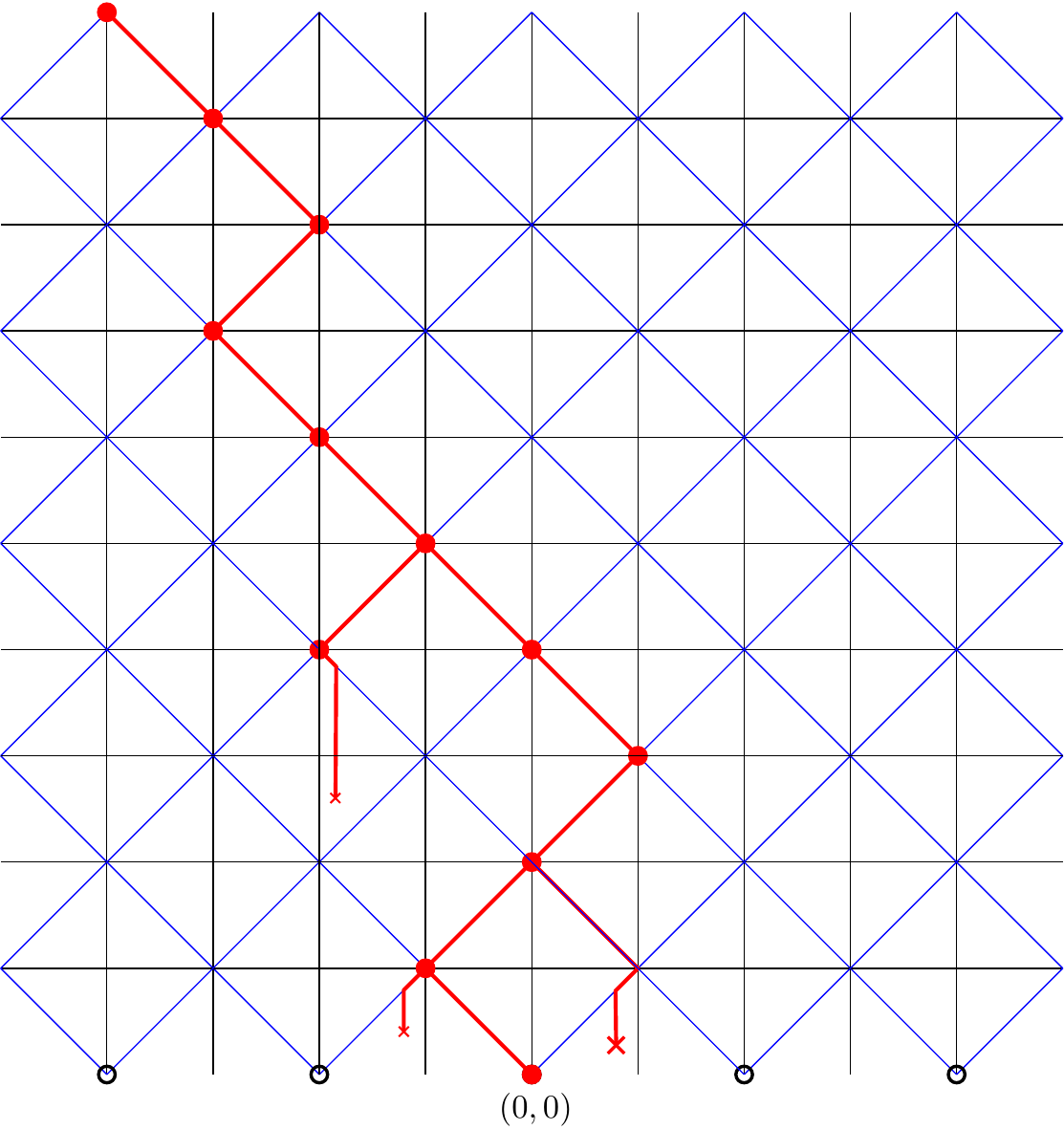}
  \caption{Colliding trajectories on the whole plane and on the lattice $\mathbb{L}_d$ (in blue).}
  \label{f:L_d}
\end{figure}

From a starting point $w\in\mathbb{L}\setminus\mathbb{L}_d$, intuitively speaking, we let $X_t^w$ follow the only path such that $(X_t^w, \pi_2(w)+t)$ remains on $\mathbb{L}$ until it hits $\mathbb{L}_d$, after which it follows the rule given by \eqref{badcompany}.
More precisely, given $w \in \mathbb{L} \setminus \mathbb{L}_d$, for any $s>0$, note that $w+s\big((-1)^{k},1\big)\in\mathbb{L}$ where $k=k(w)=\lfloor \pi_1(w)\rfloor+\lfloor \pi_2(w)\rfloor$ and define
\begin{equation}
t_0=\min\big\{s\ge 0: w+s((-1)^{k},1)\in\mathbb{L}_d\big\}.
\end{equation}
Then we define
\begin{equation}
\label{horriblecompany}
X^w_t =
\begin{cases}
\pi_1(w)+(-1)^{k}t  & \text{ if $0\le t < t_0$};\\
X^{(X^w_{t_0},\pi_2(w)+t_0)}_{t-t_0} & \text{ if $t\ge t_0$}.
\end{cases}
\end{equation}

It remains to construct the random walks starting from points $w=(x,t)\in\mathbb{R}^2 \setminus \mathbb{L}$.
The idea is very simple: it is going to be the trajectory such that $(X^w_t,\pi_2(w)+t)$ move up along direction $(0,1)$ until hitting $\mathbb{L}$ and, from this point on, follow the corresponding trajectory in $\mathbb{L}$ as defined in \eqref{horriblecompany}.

More precisely, consider $w=(x,t)\in\mathbb{R}^2$ and let
\begin{equation}
s_0=\min\{s\ge 0: (x,t+s)\in\mathbb{L}\}.
\end{equation}
Then we define
\begin{equation}
\label{awfulcompany}
X^w_s =
\begin{cases}
x &  \text{ for all $0\le s < s_0$};\\
X^{(x,\pi_2(w)+s_0)}_{s-s_0} &\text{ for all $s\ge s_0$}.
\end{cases}
\end{equation}

Equations \eqref{badcompany}, \eqref{horriblecompany} and \eqref{awfulcompany} define the coupled family $(X^w_t,w\in\mathbb{R}^2,t\in \mathbb{R}_+)$, such that the points $(X^w_t, \pi_2(w)+t)$ always remain on $\mathbb{L}$, after the first time they hit $\mathbb{L}$.

It is important for us that, for any $w\in\mathbb{R}^2$ and for any $t\ge s\ge 0$, we have
\begin{equation}\label{e:lipschitz}
  \left|X^w_t - X^w_s\right| \le t - s.
\end{equation}

\begin{remark}\label{r:shift}
It should be noted that the law of $(X_t^w)_{t\ge0}$, for $w\in\mathbb{R}^2$ is not invariant under shifts of $w$. Nevertheless, the processes $(X_t^w)_{t\ge0}$, $(X_t^{w+(1,1)})_{t\ge0}$ and $(X_t^{w+(-1,1)})_{t\ge0}$ have the same distribution.
Therefore, the law of the collection $\{(X_t^w)_{t\ge0} \colon w\in\mathbb{R}^2\}$, it fully determined by the law of $\{(X_t^w)_{t\geq 0} \colon w\in\mathcal{L}_1\}$, where
\begin{equation}\label{e:losange}
\mathcal{L}_1:=\{w \in\mathbb{R}^2:\ |\pi_1(w)|+|\pi_2(w)|\le1\}.
\end{equation}
\end{remark}

Finally, we will keep the notation $P^\eta_{p_\bullet,p_\circ}$ and $\mathbb{P}^\rho_{p_\bullet,p_\circ}$ or simply $P^\eta$ and $\mathbb{P}^\rho$, for the quenched and annealed joint laws of the family of random walks $(X^w_t,t\ge0,w\in\mathbb{R}^2)$, respectively.

The result below  states a useful monotonicity property for the collection of random walks defined above.
In the following, for $s,t\in\mathbb{R}$, we denote $s\vee t=\max(s,t)$.

\begin{proposition}\label{p:monotone}
For every $\rho\in(0,1)$, every $z,z'\in\mathbb{R}^2$ with $\pi_1(z')\le \pi_1(z)$ and $\pi_2(z)= \pi_2(z')$, we have that, almost surely,
\begin{equation}\label{e:monotone1}
  X^{z'}_{t}\le X^{z}_{t}\text{ for all }t\ge0.
\end{equation}
In fact, for every $z,z'\in \mathbb{R}^2$  such that $\pi_1(z')\le \pi_1(z)-\left|\pi_2(z')-\pi_2(z)\right|$,
\begin{equation}\label{e:monotone2}
  X^{z'}_{t+[(\pi_2(z)-\pi_2(z'))\vee0]}\le X^{z}_{t+[(\pi_2(z')-\pi_2(z))\vee0]}\text{ for all }t\ge0.
\end{equation}
\end{proposition}

\begin{figure}[h]
  \center
  \includegraphics[scale=1]{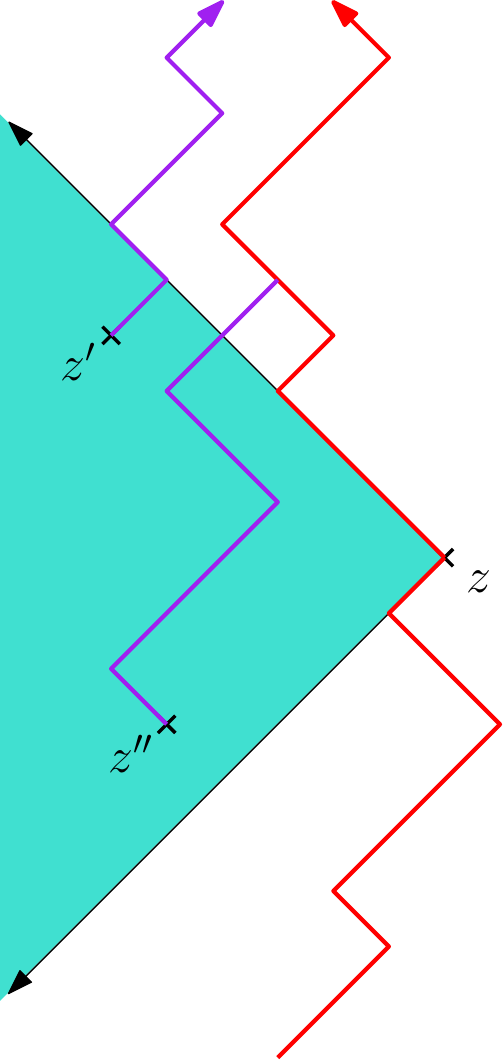}
  \caption{Proposition \ref{p:monotone} states that any trajectory started in the turquoise quadrant, e.g.~from $z'$ or $z''$, stays on the left of the trajectory started at $z$.}
  \label{f:Lip_plus}
\end{figure}

\begin{proof}
  First, we prove \eqref{e:monotone1}. This is a simple consequence of the fact that the two walks $X^{z}$ and $X^{z'}$ evolve in continuous time and space, and, as they start at points $z$ and $z'$ with same time coordinates ($\pi_2(z)= \pi_2(z')$), they cannot cross each other without being at the same position, i.e.~either $X^{z'}_{t}< X^{z}_{t}$ for all $t\ge0$, or there exists $t_0$ such that $X^{z'}_{t_0}= X^{z}_{t_0}$ and then, by construction, $X^{z'}_{t_0+s}= X^{z}_{t_0+s}$ for all $s\ge0$.

  Second, we prove \eqref{e:monotone2}.
  Assume first that $\pi_2(z')\ge\pi_2(z)$ and $\pi_1(z')\le \pi_1(z)-\left(\pi_2(z')-\pi_2(z)\right)$.
  By \eqref{e:lipschitz}, we have that
  \begin{equation}
    X^z_{\pi_2(z')-\pi_2(z)}\ge \pi_1(z)-\left(\pi_2(z')-\pi_2(z)\right)\ge \pi_1(z')= X^{z'}_0.
  \end{equation}
  The conclusion then follows from \eqref{e:monotone1}.

  Similarly, if $\pi_2(z')\le\pi_2(z)$ and $\pi_1(z')\le \pi_1(z)-\left(\pi_2(z)-\pi_2(z')\right)$. By \eqref{e:lipschitz}, we have that
  \begin{equation}
    X^{z'}_{\pi_2(z)-\pi_2(z')}\le \pi_1(z')+\left(\pi_2(z)-\pi_2(z')\right)\leq \pi_1(z) = X^{z}_0.
  \end{equation}
  The conclusion again follows from \eqref{e:monotone1}.
\end{proof}

\subsection{Structure of the proof}

In this section, we state the main propositions that lead to the proof of Theorem \ref{t:speed} and Theorem~\ref{l:exclusion_symmetric}. For this purpose, we need some more definitions.

Let us define the important event $A_{H, w}(v)$.
Intuitively speaking this event indicates that a random walk starting at $w$ after time $H$ had an average speed higher than $v$.
More precisely, for any $w \in \mathbb{R}^2$, any $H \in \mathbb{R}_+$ and $v \in \mathbb{R}$ let
\begin{equation}
  \label{e:A_m_v}
  A_{H,w}(v) := \Big[ \text{there exists $y \in \big(w+[0,H)\times\{0\}\big) $ s.t.\ $X^y_{H} - \pi_1(y) \geq v H$} \Big].
\end{equation}

In several places in the paper, we will bound the probability of $ A_{H,w}(v)$. It should be noted that the probability of this event depends on $w\in\mathbb{R}^2$, because both $ A_{H,w}(v)$ and the law of $X^w$ are not invariant when we shift $w$ by an arbitrary value in $\mathbb{R}^2$ (although they are invariant with respect to shifts in the lattice $\mathbb{L}_d$).
Nevertheless, by Remark \ref{r:shift}, it is enough to consider $w\in\mathcal{L}_1$ (defined in \ref{e:losange}).

We can now safely define
\begin{equation}
  \label{e:p_k}
  p_H(v, \rho, p_\bullet, p_\circ) :=\sup_{w\in \mathbb{R}^2}  \mathbb{P}^{\rho}_{p_\bullet, p_\circ} \big(A_{H,w}(v) \big)= \sup_{w\in \mathcal{L}_1}  \mathbb{P}^{\rho}_{p_\bullet, p_\circ} \big(A_{H,w}(v) \big).
\end{equation}
When there is no risk of confusion, we will write $  p_H(v)=  p_H(v, \rho)=  p_H(v, \rho, p_\bullet, p_\circ)$.

\begin{figure}
  \centering
    \begin{tikzpicture}[use Hobby shortcut]
      \draw (0, 0) rectangle (6, 2);
      \draw[thick] (2, 0) -- (4, 0);
      \draw[thick] (2, -.15) -- (2, .15);
      \draw[thick] (4, -.15) -- (4, .15);
      \draw (2.5, 0) .. (2.8, .4) .. (3.5, .8) .. (3.5, 1.2) .. (4, 1.6) .. (4.5, 2);
      \draw[dashed] (2.5, 0) -- (4.2, 2);
      \draw[fill, below right] (2.5, 0) circle (.05) node {$y = Y^y_0$};
      \draw[fill, above right] (4.5, 2) circle (.05) node {$Y^y_{H}$};
      \draw[fill, above left] (4.2, 2) circle (.05) node {$y + H (v, 1)$};
      \draw[right] (6, .5) node {$H$};
      \draw [decorate,decoration={brace,amplitude=10pt}] (6, -.5) -- (0, -.5) node [black, midway, yshift=-0.6cm] {$3 H$};
    \end{tikzpicture}
  \caption{An illustration of the event $A_{H, 0}(v)$.
  Starting from the point $y \in \big([0,H)\times\{0\}\big) \cap \mathbb{L}$ the walker attains an average speed larger than $v$ during the time interval $[0,H]$. Picture taken from \cite{BHT18}.}
  \label{f:event_A}
\end{figure}
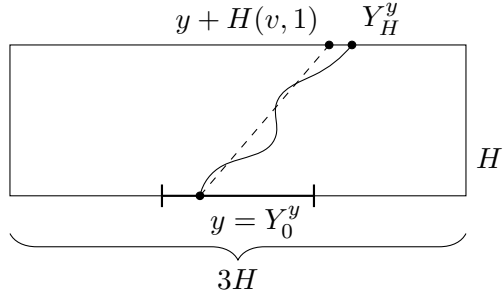
The following quantity is always well-defined and is key to our proofs of the main results.
  \begin{equation}
    \label{e:v_+}
    v_+(\rho, p_\bullet, p_\circ) := \inf \big\{ v \in \mathbb{R} \colon \liminf_{H \to \infty} p_H (v, \rho, p_\bullet, p_\circ) = 0 \big\}.
  \end{equation}
  Again, when there is no risk of confusion, we will write $ v_+=v_+(\rho)= v_+(\rho, p_\bullet, p_\circ)$. This quantity could be called  the {\it upper speed} of $X$. Indeed, for any $v>v_+$, it is unlikely that $X_t \ge vt$ for a growing sequence of $t$'s.
  On the other hand, if $v < v_+$, then $X_t \ge vt$, with probability bounded from below.

  Similarly, we define, for $w\in\mathbb{R}^2$, $v\in\mathbb{R}$ and $H\in\mathbb{R}_+$,
\begin{equation}
  \label{e:tilde_A}
  \tilde{A}_{H,w}(v):=\Big[\text{there exists $y \in \big(w+[0,H)\times\{0\}\big)$ with $X^y_{H} - \pi_1(y) \leq v H$}\Big].
\end{equation}
as well as
\begin{equation}\label{e:tildep_H}
  \tilde{p}_H(v, \rho, p_\bullet, p_\circ) :=    \sup_{w\in\mathbb{R}^2}\mathbb{P}^{\rho}_{p_\bullet, p_\circ} \big(\tilde{A}_{H,w}(v) \big)=\sup_{w\in\mathcal{L}}\mathbb{P}^{\rho}_{p_\bullet, p_\circ} \big(\tilde{A}_{H,w}(v) \big).
\end{equation}
We also define the {\it lower speed} of $X$ as
\begin{equation}
  \label{e:v_-}
  v_-(\rho, p_\bullet, p_\circ) := \sup \big\{ v \in \mathbb{R} \colon \liminf_{H \to \infty} \tilde{p}_H(v, \rho, p_\bullet, p_\circ) = 0 \big\},
\end{equation}
When there is no risk of confusion, we will write $ \tilde{p}_H(v)= \tilde{p}_H(v, \rho)= \tilde{p}_H(v, \rho, p_\bullet, p_\circ)$ and $ v_-=v_-(\rho)= v_-(\rho, p_\bullet, p_\circ)$.

Note that, as \eqref{e:probassump} is assumed, we have that, for $0\le p_\circ\le p_\bullet\le 1$, the functions $\rho\mapsto v_+(\rho,p_\bullet,p_\circ)$ and $\rho\mapsto v_-(\rho,p_\bullet,p_\circ)$ are non-decreasing.\\
Let us emphasize that
\begin{equation}
v_-, v_+ \in [-1, 1],
\end{equation}
by \eqref{e:lipschitz}, but it is a priori not guaranteed that $v_-\le v_+$. As we will see, this is in fact one consequence of the next lemma, see Corollary~\ref{l:v_+<=v_-}.

Roughly speaking, the next lemma states that, the probability of the random walk to deviate above $v_+(\rho)\vee 0$ or below $v_-(\rho)\wedge 0$ over large time scales decays very fast.

\nc{c:deviation}
\begin{lemma}
  \label{l:deviation_interp}
For every $\epsilon > 0$, there exists a constant $\uc{c:deviation}=\uc{c:deviation}(\epsilon, \rho)$ such that
  \begin{equation}
  \label{e:deviation_interp}
    \begin{split}
      & p_H((v_+(\rho)\vee0) + \epsilon, \rho) \leq \uc{c:deviation}\exp\big(-2\ln^{3/2}H\big)\\
      & \tilde{p}_H ((v_-(\rho)\wedge0) - \epsilon, \rho)\leq \uc{c:deviation}\exp\big(-2\ln^{3/2}H \big),
    \end{split}
  \end{equation}
  for every $H\ge1$.
\end{lemma}
\begin{remark}
We will present only the proof of the first inequality of \eqref{e:deviation_interp}, involving $v_+$. Nevertheless, a symmetric argument can be used to prove the second inequality.
\end{remark}
The next result, whose proof is exposed in Section \ref{s:deviations} is a simple consequence of Lemma \ref{l:deviation_interp} and will be important in the rest of the paper.

\begin{corollary}\label{l:v_+<=v_-}
{We have that $v_-(\rho)\le v_+(\rho)$, for every $\rho\in(0,1)$.}
\end{corollary}

\todo{We need to outline here when we use lateral decoupling and when we use sprinkling. This makes it easier for the reader to understand the generality of the article.}

The next result show that the two quantities $v_+$ and $v_-$ coincide, and thus identifies the candidate for the speed appearing in the LLN.

\begin{theorem}\label{l:v_+=v_-}
We have that, for any $\rho\in(0,1)$,
\begin{equation}
v_+(\rho)=v_-(\rho).
\end{equation}
\end{theorem}

Having Theorem \ref{l:v_+=v_-} in hands we can define
\begin{equation}\label{e:defv}
v(\rho) := v_+(\rho) = v_-(\rho).
\end{equation}
Combining Lemma \ref{l:deviation_interp} with Theorem \ref{l:v_+=v_-} we can derive some immediate conclusions:

\begin{theorem}
\label{t:speed_zero}
Assume that for some $\rho$, we have $v(\rho)=0$.
Then
\[
\frac{X_n}{n} \to 0, \,\,\,\,\,\, \mathbb{P}^{\rho}-\text{a.s.}
\]
\end{theorem}

\begin{proof}[Proof of Theorem \ref{t:speed_zero}]
This is a direct consequence of Lemma  \ref{l:deviation_interp}, Theorem \ref{l:v_+=v_-} and Borel-Cantelli Lemma.
\end{proof}

Furthermore, by  Theorem \ref{l:v_+=v_-} and by definition of $v_+(\rho)$ and $v_-(\rho)$, if the speed exists, then it has to be equal to $v(\rho)$. Hence, $v(\rho)$ is going to be the limiting speed appearing in Theorem \ref{t:speed}. But, as one can see, Lemma  \ref{l:deviation_interp} does not allow us to conclude the existence of the speed, or the CLT, for $v(\rho)>0$ (or for $v(\rho)<0$).
For example, when $v(\rho) > 0$, we know that it is very unlikely that the random walker will overcome the speed $v(\rho)$.
But it is not yet clear whether it can move slower than $v(\rho)$.

In order to prove Theorem~\ref{t:speed}, we will first use sprinkling in order to prove the following ballisticity result, which requires using results from \cite{RWonRW} and \cite{francoiss}. Recall the definition of $\rho_+$ and $\rho_-$ in \eqref{e:defrho+-}.
\begin{proposition}\label{p:ballisticity}
\nc{c:ball+} \nc{c:ball-}
For any $\epsilon>0$, we have that $v(\rho_++\epsilon)>0$ and $v(\rho_--\epsilon)<0$ and, for any $\rho>\rho_++2\epsilon$ and $\tilde{\rho}<\rho_--2\epsilon$ there exists constants $\uc{c:ball+}(\epsilon)$ and $\uc{c:ball-}(\epsilon)$ such that
\begin{equation}
  \label{e:ballisticity}
    \begin{split}
      & \tilde{p}_H (v(\rho_++\epsilon)/2, \rho)\leq \uc{c:ball+}\exp\big(-2\ln^{3/2}H \big)\\
            & p_H(v(\rho_--\epsilon)/2, \tilde{\rho}) \leq \uc{c:ball-}\exp\big(-2\ln^{3/2}H\big),
    \end{split}
  \end{equation}
  for all $H\ge1$.
\end{proposition}

Proposition \ref{p:ballisticity} is weaker than what one should expect, but it is actually enough to conclude the LLN and CLT of Theorem \ref{t:speed} by using results from \cite{francoiss} and \cite{RWonRW} who construct a renewal structure respectively the random walk on SSEP and PCRW.
%Nevertheless, we want to prove the continuity of $v(\cdot)$ on $(0,1)\setminus\{\rho_-,\rho_+\}$. This will require a slight modification of the renewal structure from $\clubsuit$ and $\clubsuit$, see Section $\clubsuit$.

\todo{add structure of the next sections.}

\section{Lateral decoupling}

In this section, we provide a very important property of the annealed law of the walk: if one observe the family of continuous space-time random walk defined in Section \ref{s:coupled_rw} in two disjoint 2-dimensional boxes at a space distance that is large compared to the square root of the time distance, then events in these two boxes are essentially independent.
Let us state this fact precisely.

As in Figure~\ref{f:lateral_decoupling}, fix $y_1, y_2\in\mathbb{R}^2$ such that $\pi_1(y_1)\le\pi_1(y_2)$, $\pi_2(y_1)=\pi_2(y_2)$, for $H\ge1$, let $B_1=y_1+[-H,0]\times[0,H]$ and $B_2=y_2+[0,H]\times[0,H]$ we define the distance $d(B^H_{y_1},B^H_{y_2}):=|\pi_1(y_1)-\pi_1(y_2)|$.
Our objective in the next proposition is to bound the dependence of what happens inside $B_1$ and $B_2$.

We say that a function $f :\mathcal{D}(\mathbb{R}_+, S^{\mathbb{Z}})$ is supported on a box $B_y^H$ if it is measurable with respect to $\sigma\big(\{\eta^\rho_t(x) \colon (x,t) \in B_y^H \cap (\mathbb{Z}\times \mathbb{R}_+)\}\big)$.

\begin{proposition}\label{p:decouple}
   \nc{c:decouple}
Consider the environment with law $\mathbf{P}^\rho_{EP}$ or $\mathbf{P}^\rho_{RW},$ with some density $\rho\in(0,1)$. Let $H\ge1$ and  $y_1, y_2\in\mathbb{R}^2$ such that $\pi_2(y_1)=\pi_2(y_2)$, $\pi_1(y_1)\le\pi_1(y_2)$  and such that
\begin{equation}
  \pi_1(y_1)-\pi_1(y_2)\ge H^{\frac{3}{4}}.
\end{equation}
Let $B_1=y_1+[-H,0]\times[0,H]$ and $B_2=y_2+[0,H]\times[0,H]$. For any non-negative functions $f_1$ and $f_2$, with $\|f_1\|_\infty,\| f_2\|_\infty\le1$, supported respectively on $B_1$ and $B_2$,  we have that
   \begin{equation}
     \label{e:decouple}
     {\rm Cov}_\rho(f_1, f_2) \leq \uc{c:decouple} e^{-H^{\frac{1}{4}}},
   \end{equation}
for some constant $\uc{c:decouple}=\uc{c:decouple}(\rho)$.
\end{proposition}
\begin{figure}[h]   \center
  \includegraphics{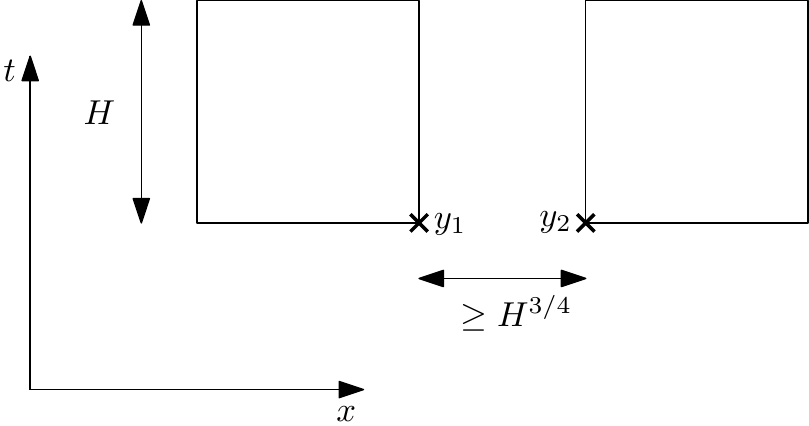}
  \caption{Lateral decoupling}
  \label{f:lateral_decoupling}
\end{figure}
\begin{proof}
The idea of the proof is that there exists an event $A$ such that $A^c$ has very small probability and $\mathbb{E}[f_1f_2\mathbf{1}_A]\le \mathbb{E}[f_1] \mathbb{E}[f_2]+2\mathbb{P}(A^c)$. In this case, as $f_1,f_2\ge0$ and   $\|f_1\|_\infty,\| f_2\|_\infty\le1$, one has that ${\rm Cov}_\rho(f_1,f_2)\le 3\mathbb{P}(A^c)$. Roughly speaking, the event $A$ will be the event that some particle visits both boxes $B_1$ and $B_2$. The proof is slightly different for SSEP and for PCRW, thus we separate the proof in two cases.\\

\noindent {\bf Case I:} The PCRW.\\
As the particules move independently for PCRW, it is clear that on the event that no particule started at time $\pi_2(y_1)$ on the right of $\pi_1(y_1)+H^{3/4}/4$ enters $B_1$ and no particule started at time $\pi_2(y_2)=\pi_2(y_1)$ on the left of $\pi_1(y_2)-H^{3/4}/4$ enters $B_2$, then the variables $f_1$ and $f_2$ behave independently, as can be shown by a simple coupling argument. Let us briefly outline a possible coupling, denoting $A$ the event described above. Let $\tilde{f}_1$ (resp.~$\tilde{f}_2$) be defined as $f_1$ (resp.~$f_2$) on an environment matching $\eta^\rho$ except that particles on the right of  $\pi_1(y_1)+H^{3/4}/4$ (resp.~left of $\pi_1(y_2)-H^{3/4}/4$) at time $\pi_2(y_1)$ are erased. The functions $\tilde{f}_1$ and $\tilde{f}_2$ are clearly independent and are respectively equal to $f_1$ and $f_2$ on the event $A$. Hence, using that $\|f_1\|_\infty,\| f_2\|_\infty\le1$ and $f_1,f_2\ge0$, on has that 
\[
\mathbb{E}[f_1f_2\mathbf{1}_A]=\mathbb{E}[\tilde{f}_1\tilde{f}_2\mathbf{1}_A]\le \mathbb{E}[\tilde{f_1}] \mathbb{E}[\tilde{f_2}]\le  \mathbb{E}[f_1] \mathbb{E}[f_2]+2 \mathbb{P}(A^c).
\]
Moreover, under $\mathbf{P}^\rho_{RW}$, $\eta^\rho_0(x)=k$ with probability $e^{-\lambda}\lambda^k/(k!)$, where we recall that $\lambda=-\ln(1-\rho)$. Hence (by Azuma's inequality for instance), we obtain, for $H\ge1$,
\begin{equation}
\begin{split}
&  {\rm Cov}_\rho  (f_1, f_2)\\
&\le 2\mathbf{P}^\rho_{RW}\left[ \exists x\in\mathbb{Z}: x> \frac{H^{3/4}}{4},1\le i \le \eta^\rho_0(x), \inf_{t\le H} \left(Y^{x,i}_t-x\right)\le -x \right]\\
&\le2 \sum_{x=\lceil H^{3/4}/4\rceil}^\infty \sum_{k\ge0} k\exp\left(-\frac{x^2}{2H}\right)e^{-\lambda}\frac{\lambda^k}{k!}\le2 \lambda e^{-\frac{H^{1/2}}{32}}\sum_{x\ge0} e^{-\frac{x^2}{2H}}\\
&\le2\lambda (1+\sqrt{2\pi H}) e^{-\frac{H^{1/2}}{32}}.
\end{split}
\end{equation}
This proves the result, choosing $\uc{c:decouple}$ properly.\\

\noindent {\bf Case II:} The SSEP.\\
In this case, the decoupling is slightly more complicated to justify, but one can do so by a coupling argument that we outline here. Consider the following construction of the environment from time $\pi_2(y_1)=\pi_2(y_2)$.
We will use two independent environment that we will couple with the actual environment until a relevant stopping time.\\
Recall the construction from Section \ref{const_ssep}. For simplicity, let us assume $\pi_2(y_1)=\pi_2(y_2)=0$ and $\pi_1(y_1)=0$.\\
 Let $\eta^{(\ell)}$ (resp. $\eta^{(r)}$) be the following environment: for $x< H^{3/4}/4$ (resp. $x> \pi_1(y_2)-H^{3/4}/4$), let $\eta^{(\ell)}_0(x)$ (resp. $\eta^{(r)}_0(x)$) be i.i.d.~Bernoulli random variables with mean $\rho$. For $x\ge H^{3/4}/4$ (resp. $x\le \pi_1(y_2)-H^{3/4}/4$), let $\eta^{(\ell)}_0(x)=g$ (resp. $\eta^{(r)}_0(x)=g$), where $g$ represents an undetermined state ( if 0 means a white particle and 1 a black particle, then $g$ could mean a green particle).\\
 Now, for every $x<H^{3/4}/2$ (resp. $x\ge H^{3/4}/2$), let $(T^{\ell,x}_i)$ (resp. $(T^{r,x}_i)$) be independent Poisson point processes with rate $\gamma$. As in Section \ref{const_ssep}, at times $(T^{\ell,x}_i)$ or $(T^{r,x}_i)$, the occupation of site $x$ and $x+1$ are exchanged. Let us denote ${\bf P}^{(\ell)}$ and ${\bf P}^{(r)}$  the laws of $\eta^{(\ell)}$ and $\eta^{(r)}$. Besides, we define $\eta^{(\ell)}$ and $\eta^{(r)}$ on a common probability space with measure $\widetilde{\bf P}$, under which we let them to be independent.\\
 Let us now define two stopping times. Let $S_1$ be the first time a green particule of $\eta^{(\ell)}$ enters $B_1$ or a particle started on the left of $H^{3/4}/4$ goes on the right of $H^{3/4}/2$, that is,
 \begin{equation}
 S_1=\inf\left\{t>0: \exists x\le 0 \text{ s.t. }  \eta^{(\ell)}_t(x)=g, \text{ or }\exists x\ge H^{3/4}/2 \text{ s.t. }  \eta^{(\ell)}_t(x)\in\{0,1\}\right\}.
 \end{equation}
 Similarly, define
 \begin{equation}
 S_2=\inf\left\{t>0: \exists x\ge \pi_1(y_2) \text{ s.t. }  \eta^{(\ell)}_t(x)=g, \text{ or }\exists x\le H^{3/4}/2 \text{ s.t. }  \eta^{(\ell)}_t(x)\in\{0,1\}\right\}.
 \end{equation}
 Also, define $S=S_1\wedge S_2$.
Let us now define $\eta^\rho$ under $\widetilde{\bf P}$, so that its marginal on $B_1\cup B_2$ corresponds to its law under ${\bf P}^\rho$. For any $t<S$, we let, for any $x\le 0$, $\eta^\rho(x)=\eta^{(\ell)}(x)$ and, for $x\ge\pi_1(y_2)$, we let $\eta^\rho(x)=\eta^{(r)}(x)$. At time $S$, each green particle takes the value of independent Bernoulli random variable with parameter $\rho$ and, after this time, the process continues as a usual SSEP as described in Section \ref{const_ssep}.\\
It is not difficult to see that, under $\widetilde{\bf P}$, the environment $\eta^{\rho}$ on $B_1\cup B_2$ has the same law as $\eta^{\rho}$ under ${\bf P}^\rho$, on $B_1\cup B_2$. Hence, $f_1$ and $f_2$ have the same law under $\widetilde{\bf P}$ and ${\bf P}^\rho$. Finally, one should note that, on the event $\{S>H\}$, $f_1$ only depends on $\eta^{(\ell)}$ and $f_2$ only depends on $\eta^{(r)}$, and thus they are conditionally independent.\\

Using the previous argument, the definition of $S$ and denoting $Y$ a continuous simple random walk with rate $\gamma$, we have that,
since $\|f_1\|_\infty \le 1$ and $\| f_2\|_\infty\le1$, 
\begin{equation}
  \begin{split}
 {\rm Cov}_\rho  (f_1, f_2)& \leq 3\widetilde{P}\left(S\le H\right) \le 12\sum_{x\ge \lceil H^{3/4}/4\rceil}\mathbf{P}\left[ \sup_{t\le H} Y_t\ge x \right]\\
& \le 12\sum_{x\ge \lceil H^{3/4}/4\rceil}\exp\left(-\frac{x^2}{2\gamma H}\right)\le 12(1+\sqrt{2\pi H})\exp\left(-\frac{H^{1/2}}{32\gamma}\right),
  \end{split}
\end{equation}
where we used similar estimated as above. This proves the result, choosing $\uc{c:decouple}$ properly.
%Under $\mathbf{P}^\rho_{EP}$, $\eta^\rho_0(x)\in\{0,1\}$, hence by the Reflection Principle and a large deviations estimate, we have that for $H \ge 512^2$,
%\todo{I don't get the $512$ here nor below...}
%\begin{equation}
%  \begin{split}
%    &\mathbf{P}^\rho_{EP}\left[ \exists x\in\mathbb{Z}: x> H+\frac{H^{3/4}}{4},1\le i \le \eta^\rho_0(x), \inf_{t\le H} \left(Y^{x,i}_t\right)\le x-\frac{H^{3/4}}{4} \right]\\
%    &\le \sum_{x=\lceil H^{3/4}/4\rceil}^\infty \exp\left(-\frac{x^2}{2H}\right)\le e^{-2H^{1/4}}.
%  \end{split}
%\end{equation}
\end{proof}

\section{Upper and lower deviations of the speed}\label{s:deviations}

This section is devoted to the proof of Lemma  \ref{l:deviation_interp} and Corollary \ref{l:v_+<=v_-}. We will prove Lemma  \ref{l:deviation_interp} for $v_+$ only but exactly the same proof, with symmetric arguments, holds for $v_-$.

\subsection{Proof of Corollaries}

We start by showing how Lemma  \ref{l:deviation_interp} implies Corollary \ref{l:v_+<=v_-}.

%\added{\begin{remark} \label{remv<1}
%The fact that the walk is nearest-neighbor yields $v_+(\rho)\le 1$.
%Also, it implies the first inequality in \eqref{e:deviation_interp} in case $v_+(\rho)= 1$.
%Similarly, the second inequality \eqref{e:deviation_interp} holds for the case $v_-(\rho)=-1$.
%Therefore, in the proofs, we can always assume
%\begin{equation}
%\label{e:bound_speeds}
%v_+(\rho)<1 \,\,\,\,\, \text{ and } \,\,\,\,\,\, v_-(\rho)>-1.
%\end{equation}
%\end{remark}
%}

\begin{proof}[Proof of Corollary \ref{l:v_+<=v_-}]
First note that, by the definition of $v_+(\rho)$ and $v_-(\rho)$, we have that, for any $\epsilon>0$,
\begin{display}
there exist two increasing sequences $(H_i^+)_i$ and $(H_i^-)_i$ such that $p_{H_i^+}(v_+(\rho)+\epsilon,\rho)<1/2$ and $\tilde{p}_{H_i^-}(v_-(\rho)-\epsilon,\rho)<1/2$.
\end{display}
Note also that, for any $v_1,v_2\in\mathbb{R}$ such that $v_1<v_2$ and any $H>0$,
\begin{display}
\label{e:pH+pH}
$p_{H}(v_1,\rho)+\tilde{p}_{H}(v_2,\rho)\ge1$.
\end{display}

We start by showing that either $v_+(\rho)\ge0$ or $v_-(\rho) \le 0$.
Indeed, assume that $v_+(\rho)<0$ and $v_-(\rho)>0$ and fix any $\epsilon\in(0,v_-(\rho)/4)$.
Lemma \ref{l:deviation_interp} implies that, for $H \in \mathbb{N}$ large enough,
\begin{equation}
  \label{e:less_than_half}
  p_{H}(\epsilon,\rho) < 1/2.
\end{equation}
Since $\epsilon<v_-(\rho)-\epsilon$, we obtain from \eqref{e:pH+pH} and \eqref{e:less_than_half} that $p_{H_i^-}(\epsilon, \rho) + \tilde{p}_{H_i^-}(v_-(\rho) - \epsilon, \rho) < 1$, as soon as $i$ is large enough.
This contradicts \eqref{e:pH+pH}.

Now consider the case $v_+(\rho)\ge0$.
Assume, by contradiction, $v_-(\rho)>v_+(\rho)$.
Fix $\epsilon\in(0,(v_-(\rho)-v_+(\rho))/4)$ so that $v_+(\rho)+\epsilon<v_-(\rho)-\epsilon$.
By  Lemma  \ref{l:deviation_interp}, for any $H\in\mathbb{N}$ large enough, $p_{H}(v_+(\rho)+\epsilon,\rho)<1/2$.
Thus, as soon as $i$ is large enough, we obtain from \eqref{e:pH+pH} that $p_{H_i^-}(v_+(\rho) + \epsilon,\rho) + \tilde{p}_{H_i^-}(v_-(\rho) - \epsilon,\rho) < 1$, which contradicts \eqref{e:pH+pH} once more.
Thus, $v_+(\rho)\ge0$ implies $v_-(\rho)\le v_+(\rho)$.

By a symmetric argument, $v_-(\rho)\le0$ implies $v_-(\rho)\le v_+(\rho)$.
This completes the proof that $v_-(\rho)\le v_+(\rho)$.
\end{proof}

Now, we prove that Theorem \ref{l:v_+=v_-} and Theorem \ref{t:speed_zero} imply Theorem~\ref{l:exclusion_symmetric}, stating that the random walk on the Exclusion process with density $1/2$ has zero speed when $p_\circ=1-p_\bullet$.

\begin{proof}[Proof of Theorem~\ref{l:exclusion_symmetric}]

Note that the exclusion process with $\rho=1/2$ is invariant in law under flipping colors $\bullet \leftrightarrow \circ$.
Thus, for any $p, q \in [0,1]$ we have $\mathbb{P}^{1/2}_{p, q} = \mathbb{P}^{1/2}_{q, p}$ which implies
\begin{equation}
\label{e:v_+flip}
v_+(1/2,p,q) = v_+(1/2,q, p).
\end{equation}

Furthermore, in the particular case $q = 1-p$ we have, for any $\rho \geq 0$, any $y\in \mathbb{L}$ and any Borel set $A \in \mathbb{R}$,
\begin{equation}
\mathbb{P}^{\rho}_{p, q} \big[ X^y_H - \pi_1(y) \in A\big] = \mathbb{P}^{\rho}_{q, p} \big[- \big(X^y_H - \pi_1(y)\big) \in A \big].
\end{equation}
In particular,
\begin{equation}
\mathbb{P}^{\rho}_{p, q} [ A_{H,w}(-v)] = \mathbb{P}^{\rho}_{q, p} [\tilde{A}_{H,w}(v)].
\end{equation}
Therefore, still assuming that $q = 1-p$ we get
\begin{equation}
\label{e:v_+=-v_-}
\begin{split}
v_+(\rho, p, q ) &= \inf\{ v\in \mathbb{R}\colon \liminf_{H\to \infty} \sup_{w\in [0,1)\times \{0\}} \mathbb{P}^{\rho}_{p, q} (A_{H,w}(v)) =0\}\\
&= -\sup\{ v\in \mathbb{R}\colon \liminf_{H\to \infty} \sup_{w\in [0,1)\times \{0\}} \mathbb{P}^{\rho}_{p, q} (A_{H,w}(-v)) =0\}\\
&= -\sup\{ v\in \mathbb{R}\colon \liminf_{H\to \infty} \sup_{w\in [0,1)\times \{0\}} \mathbb{P}^{\rho}_{q, p} (\tilde{A}_{H,w}(v)) =0\} \\
&= -v_-(\rho, q, p).
\end{split}
\end{equation}

Combining \eqref{e:v_+=-v_-} and \eqref{e:v_+flip} we get that whenever $p_\circ = 1-p_\bullet$,
\begin{equation}
v_+(1/2,p_\bullet, p_\circ) = - v_- (1/2, p_\bullet, p_\circ),
\end{equation}
and thus, by Theorem \ref{l:v_+=v_-}, we have that $v_+(1/2,p_\bullet, p_\circ)=v_-(1/2,p_\bullet, p_\circ)=0$.
We can then conclude using Theorem \ref{t:speed_zero}.
\end{proof}

\subsection{Scales and boxes}\label{s:scales}

In this section, we define some scales and boxes on $\mathbb{R}^2$ that we will use in several renormalization schemes throughout the paper.\\
Define recursively
\begin{equation}
  \label{e:L_k}
  L_0 := 10^{10} \quad \text{and} \quad L_{k + 1} := l_k L_k \text{ for $k \geq 0$, where $l_k := \lfloor L_k^{1/4} \rfloor$}.
\end{equation}

\nc{c:scale_round}
There exists $\uc{c:scale_round} > 0$ such that
\begin{equation}
  \label{e:scale_round}
  \uc{c:scale_round} L_k^{5/4} \leq L_{k + 1} \leq L_k^{5/4}, \text{ for every $k \geq 0$}.
\end{equation}
For  $L \geq 1$ and $h \geq 1$, define
\begin{equation}
  \label{e:B_L_h}
  B^h_L := [- h L, 2 h L) \times [0, h L) \subseteq \mathbb{R}^2,
\end{equation}
\begin{equation}
  \label{e:I_L_h}
  I^h_L := [0, h L) \times \{0\} \subseteq \mathbb{R}^2,
\end{equation}
and, for $w \in \mathbb{R}^2$,
\begin{equation}\label{e:B_L_h_w}
    B^h_L(w) := w + B^h_L \,\, \text{ and }\,\,
    I^h_L(w) := w + I^h_L.
\end{equation}
\begin{figure}[h]   \center
\includegraphics{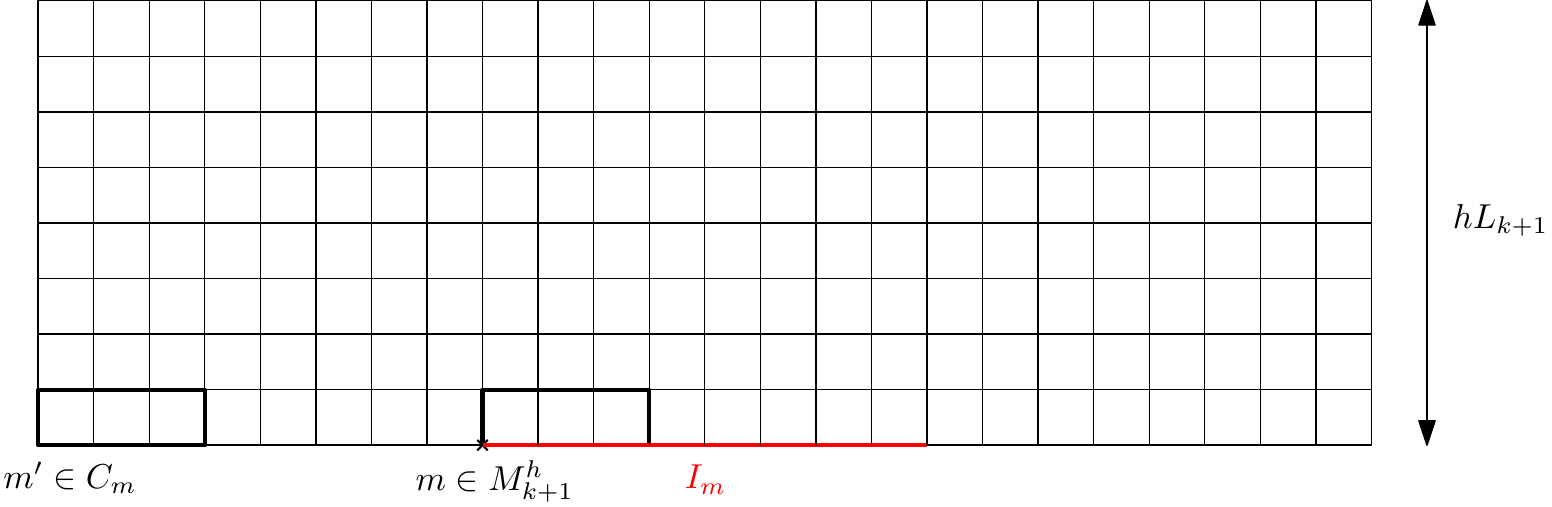}
   \caption{A box $B_m$ with $m\in M^h_{k+1}$, paved using $C_m$}
\end{figure}
It will be convenient to define the following set of indices:
\begin{equation}
  \label{e:M_h_k}
  M^h_k := \{h\} \times \{k\} \times \mathbb{R}^2,
\end{equation}
such that $m\in M^h_k$ is of the form $m=(h,k,w)$, and where we moreover define, for $v\in\mathbb{R}$,
\begin{equation}
B_m := B^h_{L_k}(w)\,,\quad I_m := I^h_{L_k}(w)\quad \text{and}\quad A_m(v):=A_{hL_k,w}(v).
\end{equation}
Note that, for each $m=(h,k,w) \in M_k^h$, a random walk starting at $I_m$ stays inside $B_m$, using \eqref{e:lipschitz}.
We also define the horizontal distance between $m = (h,k,(x,t))$ and $m' = (h,k,(x',t'))$ in $M_k^h$ as
\begin{equation}
  \label{e:d_s}
  d_s(m, m') = |x-x'|.
\end{equation}
Later on, for $m\in M^h_{k+1}$, we will want to tile the box $B_m$ with boxes $B_{m'}$ with $m'\in M^h_k$. For this purpose, we define, for $m\in M^h_{k+1}$ with $m=(h,k+1,(z,t))$,
\[
C_m=\big\{\big( h, k , (z+xhL_k, t+yhL_k) \big)\in M^h_k: (x,y)\in[-l_k,2l_k-1]\times [0,l_k-1] \cap \mathbb{Z}^2\big\}.
\]
Note that
\begin{equation}
  \label{e:size_C_m}
  |C_m| \leq 3 \ell_k^2.
\end{equation}

\subsection{A recursive inequality}
The following proposition will be used several times in the paper and is a basis of our renormalization argument.

It is important to notice that $v_{\min}$ and $v_{\max}$ in the next proposition are free parameters that we choose in different ways throughout the text.
In particular they are not related to $v_-$ and $v_+$ introduced earlier.

\nck{sds}
\nck{k0}
\begin{proposition} \label{metalemma}
Fix $0<v_{\min}<v_{\max}\le 2$. Let $\uck{k0} = \uck{k0}(v_{\min})$ be such that $\ell_{\uck{k0}}\ge( 5/v_{\min})^2$. For any $k\ge \uck{k0}$ and for all $h\ge1$, we have that
\begin{equation}
p_{hL_{k+1}}\left(v_{\min}+\frac{v_{\max}-v_{\min}}{\sqrt{l_k}}\right)\le 9\ell_k^4\left(p_{hL_k}(v_{\max})+\left(p_{hL_k}(v_{\min})\right)^2+\uc{c:decouple}e^{-\left(hL_k\right)^{1/4}}\right).
\end{equation}
\end{proposition}

The above proposition relates the probability of a speed-up event at scale $k + 1$ with similar probabilities at scale $k$.
But before proving Proposition \ref{metalemma}, prove the following deterministic lemma.
\begin{lemma}\label{metalemmadet}
  Fix $0 < v_{\min} < v_{\max} \le 2$.
  Let $\uck{k0} = \uck{k0}(v_{\min})$ be such that $\ell_{\uck{k0}}\ge( 6/v_{\min})^2$. For any $k\ge \uck{k0}$, for all $h\ge1$ and for all $m\in M^h_{k+1}$, we have that at least one of the following events happens:
\begin{itemize}
\item[$a)$] There exists $m'\in C_m$ such that $A_{m'}(v_{\max})$ occurs;
\item[$b)$] There exist $m',m''\in C_m$ such that $d_s(m',m'')\ge 4hL_k$ and such that the event $A_{m'}(v_{\min})\cap A_{m''}(v_{\min})$ occurs;
\item[$c)$]  $A^c_{m}\left(v_{\min}+\frac{v_{\max}-v_{\min}}{\sqrt{\ell_k}}\right)$ occurs.
\end{itemize}
\end{lemma}

\begin{proof}
Assume that items $a)$ and $b)$ do not hold. Define
\begin{equation}
\mathcal{B}:=\left\{m'\in C_m:\ A_{m'}(v_{\min})\text{ holds}\right\}.
\end{equation}
For $y\in I_m$, let $m_0,m_1\in\mathcal{B}$ be the first and last indexes of $\mathcal{B}$, which are visited by $\big( X^y_t, 0 \le t \le hL_{k+1} \big)$.
More precisely, $0\le i_0\le i_1\le \ell_k-1$ such that $X^y_{j_0hL_k}\in I_{m_0}$, $X^y_{j_1hL_k}\in I_{m_1}$ with $m_0=(h,k,(i_0,j_0)L_k)$, $m_1=(h,k,(i_1,j_1)L_k)$, but
\todo{here one mentions that ``indices are visited''. Do we mean that the boxes $B_m$ are visited? Or $I_m$? \\ D: We mean $I_m$, as explained after the ``more precisely''}
\begin{equation}
\left\{m'\in\mathcal{B}: \exists t'\in[0,j_0hL_k)\cup(j_1hL_k,hL_{k+1}]\text{ such that } X^y_{t'}\in I_{m'}\right\}=\emptyset.
\end{equation}
We need to consider two cases.

\begin{figure}[h]   \center
\includegraphics[scale=.8]{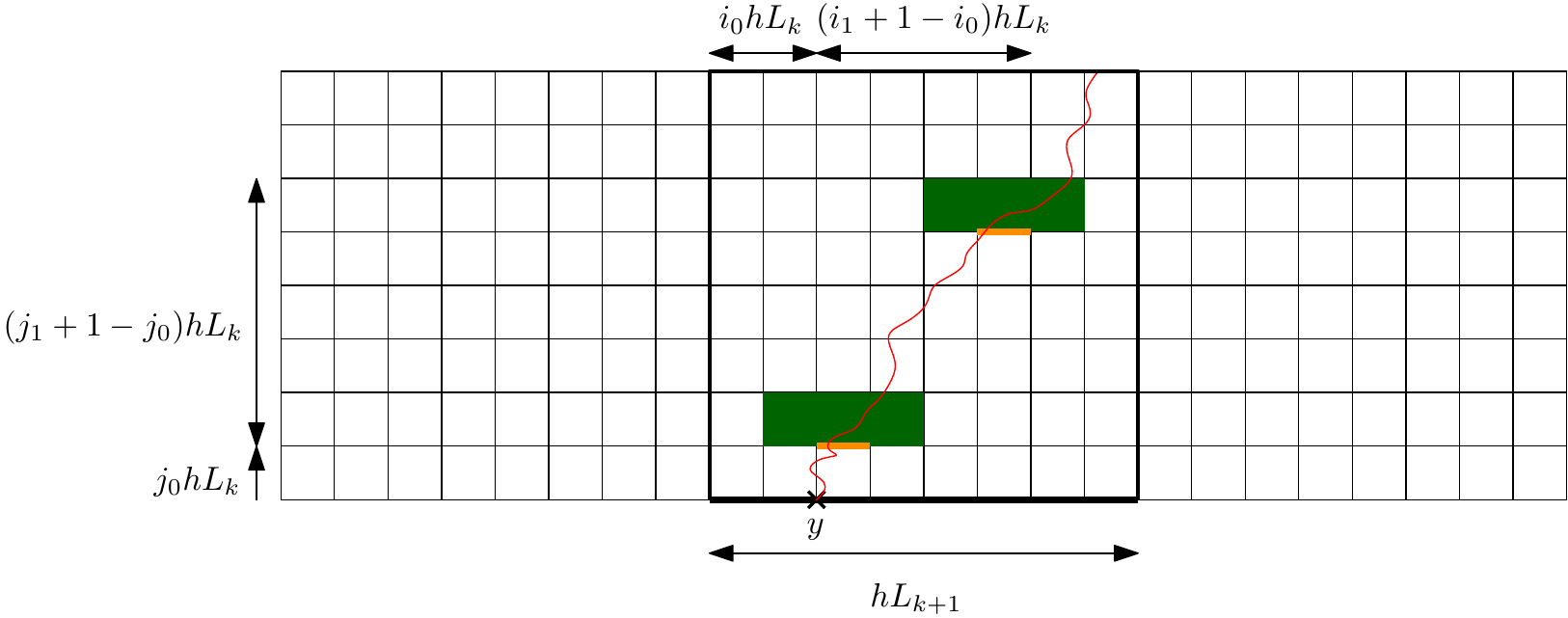}
   \caption{The points $(i_0,j_0)$ and $(i_1,j_1)$}
\end{figure}	

\noindent
{\bf Case 1: assume $j_1+1-j_0<\sqrt{\ell_k}$.}\\
As the event $a)$ does not occur, $X^y_\cdot$ moves at speed at most $v_{\max}$ between times $j_0hL_k$ and $(j_1+1)hL_k$. Moreover, by definition of $\mathcal{B}$, $j_0$ and $j_1$, $X^y_\cdot$ moves at speed at most $v_{\min}$ before time $j_0 h L_k$ and after time $(j_1 + 1) h L_k$.
Therefore, we have that
\todo{Here I changed $j_1$ by $j_1 + 1$, correct? Because between the two it could have a higher speed.\\D:yes, this looks correct to me, and in fact I add to ``+1'' elsewhere too.}
\begin{equation}
\begin{split}
X^y_{hL_{k+1}}-\pi_1(y)&\le v_{\max}\sqrt{\ell_k}hL_k+v_{\min}(\ell_k-\sqrt{\ell_k})hL_k\\
&\le hL_{k+1}\left(v_{\min}+\frac{v_{\max}-v_{\min}}{\sqrt{\ell_k}}\right).
\end{split}
\end{equation}
This implies that, in this case, $A^c_{m}\left(v_{\min}+\frac{v_{\max}-v_{\min}}{\sqrt{\ell_k}}\right)$ occurs, i.e.~item $c)$.\\

{\bf Case 2: assume $j_1+1-j_0\ge\sqrt{\ell_k}$.}\\
Again, we will use that $X^y_\cdot$ moves at speed at most $v_{\min}$ before time $j_0hL_k$ and after time $(j_1+1)hL_k$. Moreover, recall that $X^y_{j_0hL_k}\ge i_0hL_k$ and $X^y_{(j_1+1)hL_k}\le (i_1+2)hL_k$ and note that, as the event described in $b)$ does not occur, we have that $|i_0-i_1|\le 4$. This yields, as $a)$ does not occur and $6/\sqrt{\ell_k}<v_{\min}$,
\begin{equation}
\begin{split}
  X^y_{hL_{k+1}}-\pi_1(y)& \leq \left(X^y_{j_0hL_{k}}-\pi_1(y)\right)+\left( X^y_{hL_{k+1}}-X^y_{(j_1+1)hL_{k}} \right)+\left(X^y_{(j_1+1)hL_{k}}-X^y_{j_0hL_{k}}\right)\\
  & \leq v_{\min}hL_{k}\left(\ell_k-({j_1+1-j_0})\right)+6hL_k\\
  & \leq hL_{k+1}\left(v_{\min}-\frac{v_{\min}}{\sqrt{\ell_k}}+\frac{6}{\ell_k}\right).
\end{split}
\end{equation}
This implies again that $A^c_{m}\left(v_{\min}+\frac{v_{\max}-v_{\min}}{\sqrt{\ell_k}}\right)$ occurs, concluding the proof.
\end{proof}

We are now ready to prove Proposition \ref{metalemma}.

\begin{proof}[Proof of Proposition \ref{metalemma}]
  Given $m \in M_{k+1}^h$, we use Lemma~\ref{metalemmadet}, we obtain the following inclusion,
  \begin{equation}
    A_m \Big( v_{\min} + \frac{v_{\max} - v_{\min}}{\sqrt{\ell_k}} \Big) \subseteq \bigcup_{m' \in C_m} A_{m'} (v_{\max}) \cup \bigcup_{m', m'' \in C_m} A_{m'}(v_{\min}) \cap A_{m''}(v_{\min}).
  \end{equation}
  And we now have to bound the probability of the left hand side event.

  Let $m'=(h,k,(i'hL_k,t)),m''=(h,k,(i''hL_k,t))\in C_m$, with $i'\le i''$,  such that $d_s(m',m'')\ge 4hL_k$, i.e.~$i''-i'\ge4$.
  The events $A_{m'}(v_{\min})$ and $A_{m''}(v_{\min})$ are respectively supported by the boxes $((i'+2)hL_k,0)+[-hL_{k+1},0]\times[0,hL_{k+1}]$ and $((i''-1)hL_k,0)+[0,hL_{k+1}]$.

  As $(i''-i'-3)hL_{k}\ge hL_k\ge (hL_{k+1})^{3/4}$, we can apply Proposition \ref{p:decouple}.
  Recalling that $|C_m|=3\ell_k^2$, we have that
  \begin{equation}
    \begin{split}
      p_h L_{k+1} \Big( v_{\min} & + \frac{v_{\max} - v_{\min}}{\sqrt{\ell_k}} \Big) \le 3\ell_k^2 p_{hL_k}(v_{\max})\\
      & \quad + 9 \ell_k^4 \sup_{m', m'' \in C_m; d_s(m', m'') \ge 4 h L_k} \mathbb{P} \left[ A_{m'}(v_{\min}) \cap A_{m''}(v_{\min}) \right]\\
      & \overset{Proposition~\ref{p:decouple}}\le 9 \ell_k^4 \left( p_{hL_k}(v_{\max}) + p_{hL_k}(v_{\min})^2 + \uc{c:decouple} e^{-(hL_k)^{1/4}} \right)
    \end{split}
  \end{equation}
  Concluding the proof of the proposition.
\end{proof}

\subsection{Bound on $p_H(v)$}
\label{s:decay_scale}
We will first prove the following result, which states a strong decay for the probability to go faster than $(v_+\vee 0)$, along a particular subsequence of times. Once we establish this result, we will simply need to interpolate it to any value $H \ge 1$.

\begin{lemma}\label{l:deviation_reduction}
For all  $v>(v_+\vee 0)$ there exists $\uc{c:h_hat} = \uc{c:h_hat}(v) \geq 1$ and $\uck{c:k_bar}=\uck{c:k_bar}(v) \geq 1$ such that for every $k \geq \uck{c:k_bar}$
\begin{equation}
\label{e:deviation_reduction}
p_{\uc{c:h_hat}L_k}(v) \leq \exp\left({-4\left(\ln(L_{k})\right)^{3/2}}\right).
\end{equation}
\end{lemma}

\nck{c:k_bar}
\nc{c:h_hat}
From now on, we fix  $v>(v_+\vee 0)$.
Recall the definition of $\ell_k$ below \eqref{e:L_k} and let \nck{k:speed} $\uck{k:speed} = \uck{k:speed}(v)$ be such that, for all $k\ge \uck{k:speed}$,
\begin{equation}\label{ineq:ak}
{\ell_{k}}\ge\left( \left(\frac{2^{k+3}}{v-(v_+\vee0)}\right)\vee (5/v_+)\right)^2
\end{equation}
This exact choice for the constant $\uck{k:speed}$ will become clear during the proof of Lemma~\ref{l:deviation_reduction}, but for now it suffices to observe that it is well-defined because $\ell_k$ grows super-exponentially fast.
Let us define the following sequence of speeds:
\begin{equation}
 \label{e:v_k_o}
  v_{\uck{k:speed}} := \frac{v+(v_+\vee0)}{2} \quad \text{and} \quad v_{k + 1} := v_k +\frac{v-(v_+\vee0)}{2^{k+1}} \text{ for every $k \geq \uck{k:speed}$}.
\end{equation}

We have that
\begin{equation}
  v_\infty := \lim_{k \to \infty} v_k \le v_{\uck{k:speed}} + (v - (v_+\vee0))/2 = v.
\end{equation}

\begin{figure}
  \begin{center}
    \begin{tikzpicture}
      \draw (-1, 0) -- (10, 0);
      \draw (0, .2) -- (0, -.2) node [below] {$v_+$};
      \draw (3, .2) -- (3, -.2) node [below] {$v_{\uck{k:speed}}$};
      \draw (6.3, .2) -- (6.3, -.2) node [below] {$v\wedge 1$};
    %  \draw (9, .2) -- (9, -.2) node [below] {$2$};
      \draw (4, .2) -- (4, -.2) node [below] {$v_{\uck{k:speed} + 1}$};
      \draw (4.5, .2) -- (4.5, -.2);
      \draw (4.9, .2) -- (4.9, -.2);
      \draw (4.75, .2) -- (4.75, -.2) node [below] {$\dots$};
      \draw (5.5, .2) -- (5.5, -.2) node [below] {$v_{\infty}$};
    \end{tikzpicture}
  \end{center}
  \caption{The sequence of velocities $v_k$ as defined in \eqref{e:v_k_o}}
  \label{fig:v_k}
\end{figure}
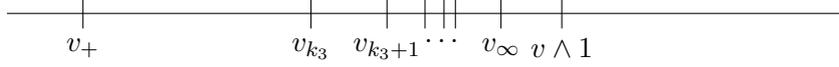

Recall the definition of $C_m$ below \eqref{e:d_s}.
We are now ready to conclude the proof of Lemma~\ref{l:deviation_reduction}.

\begin{proof}[Proof of Lemma~\ref{l:deviation_reduction}]
  Observe first that $2 > (5/4)^{3/2} \sim 1.4$, so that we can choose $\uck{c:k_bar} \ge \uck{k:speed}$ such that, for any $k\ge \uck{c:k_bar}$,
  \begin{equation}\label{e:condk0}
9 \ell_k^4 \left( e^{-4(2 - (5/4)^{3/2}) \left( \ln(L_k) \right)^{3/2}} + \uc{c:decouple} e^{-(L_k)^{1/4}+8\ln(L_k)} \right)\le1,
  \end{equation}
  where $\uc{c:decouple}=\uc{c:decouple}(\rho)$ is defined in Proposition \eqref{p:decouple}.
  Since $v_{\uck{c:k_bar}} > v_+$, we have
  \begin{equation}
    \liminf_{h \to \infty} p_{hL_{\uck{c:k_bar}}}(v_{\uck{c:k_bar}}) = 0.
  \end{equation}
  Therefore, we fix $\uc{c:h_hat}(v)\geq 1$ for which
  \begin{equation}
    p_{\uc{c:h_hat}L_{\uck{c:k_bar}}}(v_{\uck{c:k_bar}}) \leq e^{-4\left(\ln(L_{\uck{c:k_bar}})\right)^{3/2}}.
  \end{equation}
  \todo{I changed $2$ to $4$, right?\\D:Yep}
  Now we can iteratively use Proposition \ref{metalemma}, for every $k\ge \uck{c:k_bar}$, by choosing $v_{\min}=v_k$ and $v_{\max}=2$. In particular, $p_{\uc{c:h_hat}L_k}(v_{\max})=0$ and, by \eqref{ineq:ak},
  \begin{equation}
    v_{k+1}\ge v_k+\frac{2}{\sqrt{\ell_k}}\ge v_k+\frac{v_{\max}-v_{\min}}{\sqrt{\ell_k}}.
  \end{equation}
  Therefore, we can obtain the statement of the lemma through induction, by simply observing that for all $k\ge\uck{c:k_bar}$,
  \begin{equation}
    \begin{split}
      \frac{p_{\uc{c:h_hat} L_{k+1}} (v_{k+1})} {e^{-4 \left( \ln(L_{k+1}) \right)^{3/2}}} & \le 9 \ell_k^4 \left( e^{-8 \left(\ln(L_k) \right)^{3/2}} + \uc{c:decouple} e^{-(\uc{c:h_hat} L_k)^{1/4}} \right) e^{4 \left( \ln(L_{k+1}) \right)^{3/2}}\\
      & \le 9 \ell_k^4 \left( e^{-4(2 - (5/4)^{3/2}) \left( \ln(L_k) \right)^{3/2}} + \uc{c:decouple} e^{-(L_k)^{1/4}+8\ln(L_k)} \right) \le 1.
    \end{split}
  \end{equation}
  where we used \eqref{e:condk0}, the fact that $\uc{c:h_hat}\ge1$ and $L_k^{3/4}\ge L_k^{1/2}$ for $k\ge0$.
  \todo{I removed something that didn't make sense to me here...\\D:does not make sense to me either... looks better now}
\end{proof}

\subsection{Proof of Lemma~\ref{l:deviation_interp}}
\label{s:proof_deviation_interp}
With Lemma~\ref{l:deviation_reduction} at hand, we just need an interpolation argument to establish Lemma~\ref{l:deviation_interp}. Let $v=(v_+\vee0)+\epsilon$, $v' = ((v_+\vee0) + v)/2$ and let $\uc{c:h_hat}(v')$ and $\uck{c:k_bar}(v')$ be as in Lemma~\ref{l:deviation_reduction}. For $H \ge1$ let us define $\bar{k}$ as being the integer that satisfies:
\begin{equation}
 \label{e:change_scale}
  \uc{c:h_hat} L_{\bar{k}+1} \leq H < \uc{c:h_hat} L_{\bar{k} + 2}.
\end{equation}
Let us first assume that $H$ is sufficiently large so that $\bar{k} \geq \uck{c:k_bar}$, that
\begin{equation}
  \label{e:H_large}
  \frac{\uc{c:h_hat}}{\ell_{\bar{k}}} H\le \frac{v-v'}{2}H \text{, and }8L_{\bar{k}}^{9/8}e^{-2{\left(\ln(L_{\bar{k}})\right)^{3/2}}}\le 1.
\end{equation}

Therefore, we can apply Lemma~\ref{l:deviation_reduction} to conclude that
\begin{equation}\label{e:deviation_reduction2}
p_{\uc{c:h_hat}L_{\bar{k}}}(v')\leq \exp(-4{\left(\ln(L_{\bar{k}})\right)^{3/2}}).
\end{equation}

Now, in order to bound $p_H(v)$, we are going to start by fixing some $w \in\mathbb{R}^2$ and pave the box $B^1_H( w)$ with boxes $B_{m}$  with $m\in M^{\uc{c:h_hat}}_{\bar{k}}$ such that $m=(\uc{c:h_hat},\bar{k},w+(x\uc{c:h_hat}L_{\bar{k}}, y\uc{c:h_hat}L_{\bar{k}})$, where $-\lceil H/\uc{c:h_hat}L_{\bar{k}}\rceil\le x\le \lceil 2H/\uc{c:h_hat}L_{\bar{k}}\rceil$ and $0\le y \le \lceil H/\uc{c:h_hat}L_{\bar{k}}\rceil$ are integers.
Let us denote $M$ the set of such indices. Note that
\begin{equation}\label{sons}
|M|\le 8\left(\frac{H}{\uc{c:h_hat}L_{\bar{k}}}\right)^2\le 8\left(\frac{L_{\bar{k}+2}}{L_{\bar{k}}}\right)^2\le  8\left(\frac{L_{\bar{k}}^{25/16}}{L_{\bar{k}}}\right)^2\le 8 L_{\bar{k}}^{9/8}.
\end{equation}

An important observation at this point is that, on the event $\cap_{m\in M} (A_{m}(v'))^{\mathsf{c}}$, for any $y \in I^1_H( w)$ the displacement of $X^y$ up to time ${\lfloor H/\uc{c:h_hat}  L_{\bar{k}} \rfloor \uc{c:h_hat} L_{\bar{k}}}$ can be bounded by
\begin{equation}\label{e:inter1}
 \begin{split}
  X^y_{\lfloor H/\uc{c:h_hat} L_{\bar{k}} \rfloor \uc{c:h_hat} L_{\bar{k}}} - \pi_1(y)
  & = \sum_{j = 0}^{\lfloor H/\uc{c:h_hat} L_{\bar{k}} \rfloor - 1}
  X^{X^y_{j \uc{c:h_hat} L_{\bar{k}}}}_{\uc{c:h_hat} L_{\bar{k}}} - X^y_{j \uc{c:h_hat} L_{\bar{k}}} \\
  & \leq v' \lfloor H/\uc{c:h_hat} L_{\bar{k}} \rfloor \uc{c:h_hat} L_{\bar{k}} \leq v' H.
 \end{split}
\end{equation}
where we used that $A_m(v')$ does not occur for any $m \in M$ and that each point $X^y_{j \uc{c:h_hat} L_{\bar{k}}}$ belongs to $I_{m}$ for some $m' \in C_m$.
Besides, we have that
\begin{equation}\label{e:inter2}
H-{\left\lfloor \frac{H}{\uc{c:h_hat}  L_{\bar{k}}} \right\rfloor \uc{c:h_hat} L_{\bar{k}}}\le \uc{c:h_hat} L_{\bar{k}}\le \frac{\uc{c:h_hat}}{\ell_{\bar{k}}} H \overset{\eqref{e:H_large}}\le \frac{v-v'}{2}H.
\end{equation}
Therefore, by the Lipschitz condition \eqref{e:lipschitz} and \eqref{e:inter1}, on the event $\cap_{m\in M} (A_{m}(v'))^{\mathsf{c}}$, for any $y \in I^1_H( w)$,
\begin{equation}\label{e:inter3}
  X^y_H - \pi_1(y)< vH.
\end{equation}
Thus, using \eqref{e:H_large} and \eqref{sons}, this yields that
\todo{Where is this entropy estimate of $9/16...$?\\D: I added \eqref{sons}}
\begin{equation}\label{e:inter4}
  \begin{split}
    \mathbb{P} \big(A_{H,w}(v)\big)
        & \leq 8L_{\bar{k}}^{9/8}\exp(-4{\left(\ln(L_{\bar{k}})\right)^{3/2}}) \leq e^{-2\ln^{3/2}\left(H\right)}.
  \end{split}
\end{equation}
The conclusion of Lemma~\ref{l:deviation_interp} now follows by taking the suppremum over all $w \in [0,1) \times \{0\}$ and then properly choosing the constant $\uc{c:deviation}$ in order to accommodate small values of $H$.
\todo{I think this end of the argument could improve.}

This finishes the proof of Lemma~\ref{l:deviation_interp}.

\section{Proof of Theorem \ref{l:v_+=v_-}}
\label{s:threat_points}

As we discussed above, we want to show that $v_+ = v_-$.
We will assume by contradiction that $v_+>v_-$.
Then either $v_+>0$ or $v_-<0$.
We pick $v_+>0$.
The other case can be handled analogously by symmetry.

Let us define
\begin{equation}
  \label{e:delta}
\delta:=\frac{v_+ - v_-}{4}.
\end{equation}
Note $\delta \in (0, 1/2]$, since we argue by contradiction and assume that $v_+ > v_-$.

The goal of this section is to prove the following proposition which, as we show below, immediately implies Theorem \ref{l:v_+=v_-}.
\nck{kfinaldelay}  \nc{minhfinal} \nc{c:vdelay}
\begin{proposition}\label{prop:final_delay}
There exist $\uck{kfinaldelay}(\uc{c:decouple})$,  $\uc{c:vdelay}(v_+,v_-,k)=\uc{c:vdelay}(k)\ge1$ and $\uc{minhfinal}(\delta, \uck{kfinaldelay})$ such that, for all $k\ge \uck{kfinaldelay}$, for all $h\ge \uc{minhfinal}$, for all $m\in M^h_k$,
\begin{equation}\label{makesgoodforyou}
\mathbb{P}\left( A_m\left( v_+-\frac{\delta}{2\uc{c:vdelay}(\uck{kfinaldelay})} \right) \right)\le e^{-\left(\ln L_k\right)^{3/2}}.
\end{equation}
\end{proposition}

\begin{proof}[Proof of Theorem \ref{l:v_+=v_-}]
This proposition implies that there exists $\epsilon>0$ independent of $H$ such that $\liminf_{H\to\infty} p_H(v-\epsilon)=0$, which contradicts the definition of $v_+$. Therefore, this proves by contradiction that $v_+=v_-$ and thus  Theorem \ref{l:v_+=v_-}.
\end{proof}

\subsection{Trapped points}

The first step of the proof is to introduce the notion of \textit{traps}.
Intuitively speaking, by the definitions of $v_-$ and $v_+$, we know that the random walker has a reasonable probability of attaining speeds close to both of these values.
However, every time the random walker reaches a speed close to $v_-$ it will make it harder to get close to $v_+$.
Specially since it is very unlikely that it will run much faster than $v_+$ at any moment.
This motivates the definition below.

\begin{definition}
  Given $K \geq 1$ and $\delta$ as in \eqref{e:delta}, we say that a point $w \in \mathbb{R}^2$ is $K$-trapped if there exists some $y \in \big(w + [\delta K, 2 \delta K] \times \{0\} \big) \cap \mathbb{L}$ such that
  \begin{equation}
  \label{e:trapped_point}
    X^y_K - \pi_1(y) \leq (v_- + \delta)K.
  \end{equation}
  Note that this definition applies to points $w \in \mathbb{R}^2$ that do not necessarily belong to $\mathbb{L}$.
\end{definition}

As we mentioned above, the existence of a trap will introduce a delay for the random walker.
In fact, by monotonicity, if $w$ is $K$-trapped, then for every $w' \in \big( w + [0, \delta K] \times \{0\} \big) $, we have
\begin{equation}
  \label{e:trapped_slow}
  X^{w'}_K - \pi_1(w') \leq X^y_K - \pi_1(y) + 2 \delta K \leq (v_- + 3 \delta) K = (v_+ - \delta) K,
\end{equation}
where $y$ is any point in $\big(w + [\delta K, 2 \delta K] \times \{0\} \big)$ satisfying \eqref{e:trapped_point}.

Our next step is to show that the probability to find a trap in uniformly bounded away from zero.

\nc{c:some_trapped}
\nc{c:H_lower}
\begin{lemma}
  \label{l:some_trapped}
  There exist constants $\uc{c:some_trapped}(v_+,v_-) > 0$ and $\uc{c:H_lower} (v_+,v_-)> 4/\delta$, such that
  \begin{equation}
    \label{e:some_trapped}
    \inf_{K \geq \uc{c:H_lower}} \;\; \inf_{w \in \mathbb{R}^2 } \mathbb{P} \big[ \text{$w$ is $K$-trapped} \big] \geq \uc{c:some_trapped}.
  \end{equation}
\end{lemma}

\begin{proof}
Since $v_-+\delta>v_-$, the definition of $v_-$ implies the positivity of the following constant:
\begin{equation}
\uc{c:some_trapped} := \frac{1}{2} \Big\lceil \frac{2}{\delta} \Big\rceil^{-1} \liminf_{K\to\infty} \tilde{p}_{K}(v_-+\delta)>0.
\end{equation}
In particular, there exists $\uc{c:H_lower}>8/\delta$ such that
\begin{equation}
\Big\lceil \frac{2}{\delta} \Big\rceil^{-1} \inf_{K\geq \uc{c:H_lower}-8/\delta} \tilde{p}_K(v_-+\delta) \geq \uc{c:some_trapped}.
\end{equation}
If we had a supremum over $w\in\mathbb{R}^2$ in \eqref{e:some_trapped}, we would be done.
However, we have an infimum in \eqref{e:some_trapped}, so that the proof requires a few more steps.

Recall the definition of $\mathcal{L}_1$ in \eqref{e:losange}. Let us prove that if
\begin{equation}\label{e:claim1}
\begin{split}
&\text{there is }z\in\mathcal{L}_1\text{ and }z'\in\left(z+[\delta K+4,2\delta K-4]\times\{0\}\right)\\
&\text{ such that }X^{z'}_K-\pi_1(z')\le (v_-+\delta)K-4,
\end{split}
\end{equation}
then
\begin{equation}\label{e:claim2}
\begin{split}
&\text{for any }y\in\mathcal{L}_1\text{ there exists }y'\in\left(y+[\delta K,2\delta K]\right)\\
&\text{ such that }X^{y'}_K-\pi_1(y')\le (v_-+\delta)K.
\end{split}
\end{equation}
Assume that \eqref{e:claim1} holds and fix $y\in \mathcal{L}_1$. Assume first that $\pi_2(y)\le \pi_2(z)$. Define
$y'=(\pi_1(z')-(\pi_2(z)-\pi_2(y)),\pi_2(y))$. Thus $\pi_2(y')\le \pi_2(z)$ and $\pi_1(y')\le \pi_1(z')-(\pi_2(z')-\pi_2(y'))$. Hence, by Proposition \ref{p:monotone}, we have that, for $K\ge2$,
\begin{equation}
X^{y'}_K\le X^{z'}_{K-(\pi_2(z')-\pi_2(y'))}.
\end{equation}
By \eqref{e:lipschitz} and using that $|\pi_2(z')-\pi_2(y')|\le2$, we have that
\begin{equation}
\begin{split}
X^{y'}_K&\le \pi_1(z')+(\pi_2(z')-\pi_2(y'))+  (v_-+\delta)K-4\\
&\le \pi_1(y')+2(\pi_2(z)-\pi_2(y))+(v_-+\delta)K-4\\
&\le\pi_1(y')+(v_-+\delta)K.
\end{split}
\end{equation}
Moreover, as $z,y\in\mathcal{L}_1$ and $\pi_1(y')=\pi_1(z')-(\pi_2(z)-\pi_2(y))$, we have that $\delta K\le \pi_1(y')\le 2\delta K$.\\

If $\pi_2(y)\ge \pi_2(z)$, similar arguments hold by defining $y'=(\pi_1(z')-(\pi_2(y)-\pi_2(z)),\pi_2(y)$.\\

Then, for $K > \uc{c:H_lower}$, let $\tilde{K}=K-8/\delta$, so that we have, using translation invariance,
  \begin{equation*}
    \begin{split}
     \uc{c:some_trapped} & \leq
        \sup_{z \in \mathcal{L}_1}
      \mathbb{P} \Big[
      \begin{array}{c}
        \text{there exists $z' \in (z + [0, \tilde{K}) \times \{0\})$}\\
         \text{such that $X^{z'}_{K} - \pi_1(z') \leq (v_- + \delta) \tilde{K}$}
      \end{array}
      \Big]\\
      & \leq  \sup_{z \in \mathcal{L}_1}
      \mathbb{P} \Big[
      \begin{array}{c}
        \text{there exists $z' \in (z + [\delta\tilde{K}, 2\delta\tilde{K}) \times \{0\})$}\\
         \text{such that $X^y_{K} - \pi_1(z') \leq (v_- + \delta) \tilde{K}$}
      \end{array}
      \Big]\\
      & \leq \inf_{y \in \mathcal{L}_1}
      \mathbb{P} \Big[
      \begin{array}{c}
        \text{there exists $y' \in (y + [\delta K, 2\delta K) \times \{0\})$}\\
         \text{ such that $X^{y'}_{K} - \pi_1(y') \leq (v_- + \delta) K$}
      \end{array}
      \Big].
    \end{split}
  \end{equation*}
  By Remark \ref{r:shift}, the infimum over $\mathcal{L}_1$ is equal to the infimum over $\mathbb{R}^2$, and we can thus conclude.
   \end{proof}

   Let us describe how we intend to employ the last lemma, which is widely inspired by that of \cite{BHT18}. The basic idea is that if a point is trapped, then a walk started from there is delayed. Then, one could argue that if a point is trapped with a probability that is high enough, then this would provide the picture of a super-critical percolation.
   In such a scenario, any random walker would have to be delayed on large distances and it would not be possible to remain close to $v_+$.

   Nevertheless, Lemma \ref{l:some_trapped} does not guarantee that the probability that a point is trapped is high. For this reason, we will use here the notion of {\it threatened points} introduced in \cite{BHT18}.
   Intuitively speaking, we will say that a point is threatened if there exists at least one trapped point lying along a line segment with slope $v_+$ starting from this point, see Definition \ref{d:threatened} and Figure~\ref{f:threatened}.\\
We will then prove two key results: a point is threatened with very high probability (see Lemma \ref{l:threatened}) and a random walk starting at a threatened point is delayed with very high probability (see Lemma \ref{l:delays2} and Lemma \ref{l:vdelay}).

\subsection{Threatened points}

\begin{definition}
  \label{d:threatened}
  Given $\delta$ as in \eqref{e:delta}, $K \geq 1$ and some integer $r \geq 1$, we say that a point $w \in \mathbb{L}$ is $(K, r)$-threatened if $w + j K (v_+, 1)$ is $K$-trapped for some $j = 0, \dots, r-1$.
\end{definition}

\begin{figure}
  \begin{center}
    \begin{tikzpicture}[use Hobby shortcut]
      \draw[dashed] (4.4, 0) -- (6.2, 3);
      \foreach \x in {1,...,5}
      { \draw[fill] (4.3 + \x / 3.3333, \x / 2) -- (4.5 + \x / 3.33, \x / 2);
      }
      \draw[fill, left] (5.3, 1.5) circle (.05) node {$\lfloor y \rfloor_{K} + j_o K(v_+, 1)$};
      \draw[dotted] (5.3, 1.5) -- (7, 1.5);
      \draw[thick] (5.6, 1.5) -- (5.9, 1.5);
      \draw[thick] (5.6, 1.45) -- (5.6, 1.55);
      \draw[thick] (5.9, 1.45) -- (5.9, 1.55);
      \draw[fill, left, right] (4.6, 0) circle (.05) node {$y = Y^y_0$};
      \draw[fill, left] (4.4, 0) circle (.05) node {$\lfloor y \rfloor_{K}$};
      \draw[fill, right] (6.2, 3) circle (.05) node {$\lfloor y \rfloor_{K} + rK (v_+, 1)$};
      \draw[fill, above left] (6, 3) circle (.05) node {$Y^y_{r K}$};
      \draw (5.7, 1.5) .. (5.72, 1.6) .. (5.6, 1.7) .. (5.4, 1.8) .. (5.35, 1.9) .. (5.3, 2);
      \draw[fill, below right] (5.72, 1.5) circle (.04) node {$y'$};
      \draw[fill, below right] (5.3, 2) circle (.04);
      \draw (4.6, 0) .. (4.5, .5) .. (5.1, 1) .. (5.4, 1.5) .. (5.2, 2) .. (5.5, 2.5) .. (6, 3);
    \end{tikzpicture}
    \caption{The point $\lfloor y \rfloor_{K}$ is $(K, r)$-threatened, since $\lfloor y \rfloor_{K} + j_o K (v_+, 1)$ is $K$-trapped. Picture taken from \cite{BHT18}.}
    \label{f:threatened}
  \end{center}
\end{figure}
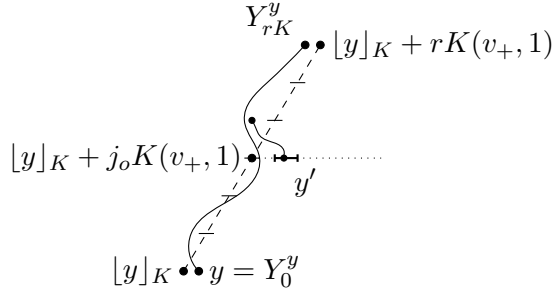

As we are going to show below, a random walker starting on a threatened point will most likely suffer a delay, similarly to what we saw for trapped points.
See Figure \ref{f:threatened} for an illustration.

\begin{lemma}
  \label{l:delays2}
  For any positive integer $r $ and any real number $K\geq \uc{c:H_lower}$, if we start the walker at some $y \in \mathbb{L}$ and there exists $w\in(y+[-\delta K/4,0])$ such that
  \begin{equation}
 \text{$w$ is $(K, r)$-threatened},
  \end{equation}
  then either
  \begin{enumerate}
  \item ``the walker runs faster than $v_+$ for some time interval of length $K$'', that is,
    \begin{equation}
      \label{e:speedup}
      X_{(j + 1)K}^y - X_{jK}^y
      \geq \Big( v_+ + \frac{\delta}{2 r} \Big)K \quad \text{ for some $j = 0, \dots, r - 1$,}
    \end{equation}
  \item or else, ``it will be delayed'', that is,
    \begin{equation}
      \label{e:delay}
      X^y_{rK} - \pi_1(y)
      \leq \Big( v_+ - \frac{\delta}{2 r} \Big) r K.
    \end{equation}
  \end{enumerate}
\end{lemma}

\begin{proof}
Fix $r \geq 1$ and $K\geq 1$.
Assume that the point $w\in(y+[-\delta K/4,0]\times\{0\})$ is $(K, r)$-threatened.
  Thus, for some $j_o \in \{ 0, \dots, r - 1\}$,  \begin{equation}
w + j_o K (v_+, 1) \text{ is $K$-trapped}
  \end{equation}
  or, in other words, there exists a point
  \begin{equation}
    \label{e:location_y'}
    y' \in \Big( \big( y + j_o K (v_+, 1) \big) + [3\delta K/4, 2 \delta K] \times \{ 0 \} \Big)
  \end{equation}
  such that
  \begin{equation}
    \label{e:delay_y'}
    X^{y'}_{ K} - \pi_1(y') \leq (v_- + \delta) K= (v_+ - 3 \delta) K.
  \end{equation}
  Fix such a point $y'$ and notice from \eqref{e:location_y'} that,
  \begin{equation}\label{e:location_y'_2}
    \frac{3}{4} \delta K \leq
    \pi_1(y') - \big( \pi_1(y) + j_o Kv_+ \big)
    \leq 2 \delta K.
  \end{equation}

We now assume that \eqref{e:speedup} does not hold and bound the horizontal displacement of the random walk in three steps: before time $j_o K$, between times $j_o K$ and $(j_o + 1) K$ and from time $(j_o + 1) K$ to time $r K$.

  \begin{equation*}
    \begin{split}
      X^y_{j_oK} - \pi_1(y)
      & \leq \sum_{j = 0}^{j_o - 1} X^y_{(j+1) K} - X^y_{j K}\\
      & \overset{\neg\eqref{e:speedup}}\leq j_o \Big( v_+ + \frac{\delta}{2 r} \Big)K
     \leq j_o v_+K + \frac{\delta}{2 } K\\
      & \leq \pi_1(y').
    \end{split}
  \end{equation*}
  So, by \eqref{e:location_y'_2}, $Y^y_{j_o K}$ lies to the left of $y'$ and, by monotonicity, \eqref{e:delay_y'} and \eqref{e:location_y'_2} we have that
  \begin{equation}
    \begin{split}
      X^y_{(j_o + 1) K} & \leq X^{y'}_{K} \leq \pi_1(y') + (v_+ - 3 \delta) K\\
      & \leq \pi_1(y) + j_o v_+ K + 2 \delta K+ (v_+ - 3 \delta) K\\
      & \leq \pi_1(y) + (j_o + 1) v_+ K - \delta K.
    \end{split}
  \end{equation}
  Now applying once more the assumption that \eqref{e:speedup} does not hold, for $j = j_o, \ldots, r-1$, we can bound the overall displacement of the random walk up to time $rK$:
  \begin{equation}
    \begin{split}
      X^y_{r K} - \pi_1(y)
      & \leq \big( X_{r K}^y - X_{(j_o + 1) K}^y \big) + \big( X_{(j_o + 1) K}^y - \pi_1(y) \big)\\
      & \leq (r - j_o - 1) \Big( v_+ + \frac{\delta}{2r} \Big) K + (j_o + 1) v_+ K - \delta K\\
      & \leq r v_+ K - \frac{\delta}{2} K = \Big( v_+ - \frac{\delta}{2 r} \Big) r K,
    \end{split}
  \end{equation}
  showing that \eqref{e:delay} holds and thus proving the result.
\end{proof}

We have seen that threatened points are introducing a delay to the walk, just like traps did.
However, the advantage of introducing the concept of threats is that they are much more likely to occur than traps as the next lemma shows.

\nc{csedar}
\nc{c:threatened}
\begin{lemma}[Threatened points]
  \label{l:threatened}
There exist $\uc{c:threatened} = \uc{c:threatened}(v_+,v_-,\uc{c:decouple})$ and $\uc{csedar}(v_+,v_-)$  such that,   for any $r \geq 1$ and for any $K\ge \uc{csedar}(v_+,v_-)$,
  \begin{equation}
  \sup_{y\in\mathbb{R}^2}  \mathbb{P} \big[ \text{$y$ is not $(K, r)$-threatened} \big] \leq \uc{c:threatened} r^{-1000}.
  \end{equation}
\end{lemma}

\begin{proof}
\nc{c:tilda_trap}
First, we prove a statement for $r=3^j$ for all integers $j \ge 3$.
Let us define
\begin{equation}
q_j^K=\sup_{y\in\mathbb{R}^2}\mathbb{P}\left[y \text{ is not }(K,3^j)\text{-threatened}\right].
\end{equation}
Note that if the event $\{y \text{ is not }(K,3^{j+1})\text{-threatened}\}$ holds for some $j\ge3$, then the two following events both hold:
\begin{equation}
\begin{split}
A_1&=\left\{y \text{ is not }(K,3^{j})\text{-threatened}\right\},\\
A_2&=\left\{y+(v_+,1)2\times3^j) \text{ is not }(K,2\times3^{j})\text{-threatened}\right\}.
\end{split}
\end{equation}
\todo{I disagree with this $2 \times$ here.}

\nck{ksedar}
These events are respectively supported on
\begin{equation}
\begin{split}
B_1&= y + \left[-K,v_+3^jK+(1+2\delta) K\right]\times\left[0,3^jK\right]\\
B_2&= y + \left[(2v_+3^j-1)K, (v_+3^{j+1}+1+2\delta) K\right] \times \left[2\times3^jK,3^{j+1}K\right].
\end{split}
\end{equation}
We want to be able to apply Proposition \ref{p:decouple}. For this purpose, we require that  $(v_+3^j-2-2\delta)K\ge (3^{j  + 1 }K)^{3/4}$ by choosing $ \uck{ksedar}(v_+)$ such that $v_+3^j-2-2\delta>v_+3^j/2$ for all $j\ge \uck{ksedar}$, and by choosing $\uc{csedar}(v_+,v_-)\ge \uc{c:H_lower}$ (defined  in Lemma \ref{l:some_trapped}) such that $v_+K^{1/4}\ge 2$ for all $K\ge\uc{csedar}(v_+)$.

\nc{csedar2}
Define the constant $\widetilde{\uc{c:some_trapped}}=1-(1-\uc{c:some_trapped})^{1/9}$, where $\uc{c:some_trapped}$ is defined  in Lemma \ref{l:some_trapped}. Let us fix $\uc{csedar2}(v_+,v_-,\uc{c:decouple})$ such that
\begin{equation}
\uc{c:decouple}\exp \Big\{-\left( (3^{\uc{csedar2}+j+1})^{3/4}-(j+1)^2\ln(1-\widetilde{\uc{c:some_trapped}}) \right) \Big\} \le\widetilde{\uc{c:some_trapped}}/2\text{, for all } j\ge3.
\end{equation}
Now, note that, for all $K\ge\uc{csedar}$, $q^K_{\uc{csedar2}+3}\le 1-\uc{c:some_trapped}\le (1-\widetilde{\uc{c:some_trapped}})^{3^2}$ by Lemma \ref{l:some_trapped}.
Assume that, for some $k\ge3$, $q^K_{\uc{csedar2}+j}\le (1-\widetilde{\uc{c:some_trapped}})^{j^2}$. Then, one has
\begin{equation}
\begin{split}
\frac{q^K_{\uc{csedar2}+j+1}}{(1-\widetilde{\uc{c:some_trapped}})^{(j+1)^2}}&\le \frac{\Big(q^K_{\uc{csedar2}+j}\Big)^{ 2 }}{(1-\widetilde{\uc{c:some_trapped}})^{(j+1)^2}}+\frac{\uc{c:decouple}e^{ -\left( 3^{\uc{csedar2}+j+1}K  \right)  }}{(1-\widetilde{\uc{c:some_trapped}})^{(j+1)^2}}\\
&\le (1-\widetilde{\uc{c:some_trapped}})^{2j^2-(j+1)^2}+\frac{\widetilde{\uc{c:some_trapped}}}{2}\le 1.
\end{split}
\end{equation}
This proves that, for any $j\ge3$, $q^K_{\uc{csedar2}+j}\le (1-\widetilde{\uc{c:some_trapped}})^{j^2}$, for all $K\ge\uc{csedar}$.
\nc{onemore}
There exists $\uc{onemore}$ such that  $ (1-\widetilde{\uc{c:some_trapped}})^{\uc{onemore}}\le 3^{-1000}$. Consequently, for any $j\ge \uc{onemore}$, $q^K_{\uc{csedar2}+j}\le 3^{-1000j}$.
Finally, let $r\ge 3^{\uc{csedar2}+\uc{onemore}}$ and let $\bar{j}$ be such that $3^{\bar{j}}\le r<3^{\bar{j}+1}$. Therefore, we have, for all $K\ge\uc{csedar}$,
\begin{equation}
\begin{split}
\sup_{y\in\mathbb{R}^2}\mathbb{P}\left[y \text{ is not }(K,r)\text{-threatened}\right]&\le \sup_{y\in\mathbb{R}^2}\mathbb{P}\left[y \text{ is not }(K,3^{\bar{j}})\text{-threatened}\right]\\
&\le (3^{\bar{j}})^{-1000}\le 3^{1000}r^{-1000}.
\end{split}
\end{equation}
By choosing
\begin{equation}
\uc{c:threatened}(v_+,v_-,\uc{c:decouple})=3^{1000(\uc{csedar2}+\uc{onemore})}
\end{equation}
the result follows.
\end{proof}

\subsection{Proof of Proposition \ref{prop:final_delay}}

We first prove the following lemma which already provides some delay compared to $v_+$. Nevertheless, this delay, when applied in Corollary \ref{l:vdelay}, is vanishing with time. Thus we will need to bootstrap this result in a following proposition, which will give a contradiction with the assumption $v_-<v_+$. Recall the definition of $A_{H,w}(v)$ in \eqref{e:A_m_v}.
\begin{lemma}\label{l:vdelay0}
For any $\epsilon>0$, there exist $r=r(\epsilon,v_+,v_-,\uc{c:decouple})\in\mathbb{R}^+$ and  $H_0=H_0(\epsilon, r, v_+,v_-)\in\mathbb{R}^+$ such that, for any $H\ge H_0$ and for any $w\in\mathbb{R}^2$,
\begin{equation}\label{e:vdelay}
\mathbb{P}\left[A_{H,w}\left(v_+-\frac{\delta}{2r}\right)\right]\le \epsilon.
\end{equation}
\end{lemma}

\begin{proof}
For $r\ge1$, $H\ge r\times\uc{csedar}$ and $w\in\mathbb{R}^2$, let us define $y_i=w+i(\delta H/4r,0)$ and the events
\begin{align}
E_1(H,r,w)&:=\left\{  \exists i\in\{0,\dots,\left\lceil\frac{4r}{\delta}\right\rceil-1\}:\ y_i\text{ is not }(\frac{H}{r},r)\text{-threatened}\right\}\\
E_2(H,r,w)&:=\left\{  \exists i\in\{0,\dots,\left\lceil\frac{4r}{\delta}\right\rceil-1\}, j\in\{0,\dots,r-1\}:\ X_{(j + 1)\frac{H}{r}}^{y_i} - X_{j\frac{H}{r}}^{y_i} \geq \left( v_+ + \frac{\delta}{2 r} \right)\frac{H}{r}\right\}.
\end{align}
Note that for any $y\in[w,w+H]$, there exists $i\in\{0,\dots,\left\lceil\frac{4r}{\delta}\right\rceil-1\}$ such that $y_i\in\left(y+\left[-\delta H/4r,0\right]\times\{0\}\right)$. Thus, by Lemma \ref{l:delays2}, we have that if $\left(E_1(H,r,y)\cup E_2(H,r,y)\right)^c$ holds then, for any $y\in[w,w+H]$, we have that
\begin{equation}\label{e:conc}
 X^y_{H} - \pi_1(y)
      \leq \Big( v_+ - \frac{\delta}{2 r} \Big) H.
\end{equation}
Therefore, we have that $A_{H,w}\left(v_+-\frac{\delta}{2r}\right)\subset E_1(H,r,y)\cup E_2(H,r,y)$, and thus
\begin{equation}\label{littleflea}
\mathbb{P}\left[ A_{H,w}\left(v_+-\frac{\delta}{2r}\right)\right]\le \mathbb{P}\left[ E_1(H,r,w)\right]+\mathbb{P}\left[ E_2(H,r,w)\right]
\end{equation}
Let us now apply Lemma \ref{l:threatened} with $K=\frac{H}{r}$. This yields that, for any $r\ge1$ and $H\ge r\times \uc{csedar}$,
  \begin{equation}\label{offthemap}
  \begin{split}
\mathbb{P}\left[ E_1(H,r,w)\right]&\le\frac{6r}{\delta} \sup_{y\in\mathbb{R}^2}  \mathbb{P} \big[ \text{$y$ is not $(\frac{H}{r}, r)$-threatened} \big] \\
&\leq \frac{6}{\delta}\uc{c:threatened}(v_+,v_-,\uc{c:decouple}) r^{-999}.
  \end{split}
  \end{equation}
  Besides, by Lemma \ref{l:deviation_interp}, we have that
    \begin{equation}\label{offthemap2}
  \begin{split}
\mathbb{P}\left[ E_2(H,r,w)\right]&\le\frac{6r^2}{\delta} p_{\frac{H}{r}}\left(v_++ \frac{\delta}{2 r}\right) \\
&\leq \frac{6r^2}{\delta}\uc{c:deviation}(v_+,v_-,r) e^{-2\left(\ln \frac{H}{r}\right)^{3/2}}.
  \end{split}
  \end{equation}
  
  Finally, notice that, for any $\epsilon>0$, we can choose $r=r(\epsilon,v_+,v_-,\uc{c:decouple})$ large enough so that the right-hand side of \eqref{offthemap} is smaller than $\epsilon/2$ for any $H\ge r\times \uc{csedar}$, and then choose $H_0=H_0(\epsilon, r, v_+,v_-)\ge r\times \uc{csedar}$ large enough so that the right-hand side of \eqref{offthemap2} is smaller than $\epsilon/2$ for any $H\ge H_0$. This concludes the proof using \eqref{littleflea}.
\end{proof}

\begin{corollary}\label{l:vdelay}
For all $k\ge0$, there exists $\uc{c:vdelay}(v_+,v_-,k)=\uc{c:vdelay}(k)\ge1$ such that, for any $h\ge \uc{c:vdelay}(k)$ and for any $w\in\mathbb{R}^2$
\begin{equation}\label{e:vdelay}
\mathbb{P}\left[A_{hL_k,w}\left(v_+-\frac{\delta}{\uc{c:vdelay}(k)}\right)\right]\le e^{-\left(\ln L_k\right)^{3/2}}.
\end{equation}
\end{corollary}

\begin{proof}
Let us fix  $k\ge0$ and apply Lemma \ref{l:vdelay0} with $\epsilon=e^{-\left(\ln L_k\right)^{3/2}}$.\\
We obtain that there exist  $r(k)=r(k,v_+,v_-,\uc{c:decouple})\ge1$ and $H_0(k,r)=H_0(k, r, v_+,v_-)$ such that, for any $h\ge H_0(k,r)/L_k$, for any $w\in\mathbb{R}^2$,
\begin{equation}
\mathbb{P}\left[A_{hL_k,w}\left(v_+-\frac{\delta}{2r(k)}\right)\right]\le e^{-\left(\ln L_k\right)^{3/2}}.
\end{equation}

  Finally, define $\uc{c:vdelay}(k)=\left(H_0(k,r)/L_k\right)\vee (2(r(k))$, so that \eqref{e:vdelay} follows by noting that $A_{hL_k,w}\left(v_+-\frac{\delta}{\uc{c:vdelay}(k)}\right)\subset A_{hL_k,w}\left(v_+-\frac{\delta}{2r(k)}\right)$ 
\end{proof}

We can now prove Proposition \ref{prop:final_delay}, which completes the proof of Theorem \ref{l:v_+=v_-}.
\begin{proof}[Proof of Proposition \ref{prop:final_delay}]
We fix $\uck{kfinaldelay}$ such that
\begin{align}\label{1star}
&\sum_{k\ge \uck{kfinaldelay}}\frac{1}{\sqrt{\ell_k}}\le \frac{1}{4},\\ \label{2star}
&9\ell_k^4\left[2e^{ -\left(2-(5/4)^{3/2}\right)\left(\ln L_k\right)^{3/2}  }+\uc{c:decouple}e^{-L_k^{1/4}+(5/4)^{3/2}\ln L_k   }\right]\le 1, \ \forall k\ge \uck{kfinaldelay}.
\end{align}
Let us fix $\uc{minhfinal}(v_+,v_-, \uck{kfinaldelay})\ge \uc{c:vdelay}(\uck{kfinaldelay})$ such that
\begin{equation}\label{3star}
\uc{c:deviation}\left(\frac{\delta}{\uc{c:vdelay}(\uck{kfinaldelay})},\rho\right)e^{-2\left(\ln \uc{minhfinal}\right)^{3/2}}\le 1,
\end{equation}
where $\uc{c:deviation}$ is defined in Lemma \ref{l:deviation_interp}.\\
Now, we want to apply Proposition \ref{metalemma}, iteratively. For this purpose, let us define
\begin{align*}
v_{\max}=v_++\frac{\delta}{\uc{c:vdelay}(\uck{kfinaldelay})},\ v_{\uck{kfinaldelay}}=v_+-\frac{\delta}{\uc{c:vdelay}(\uck{kfinaldelay})},
\end{align*}
and for all $k\ge \uck{kfinaldelay}$,
\begin{equation}\label{defvkend}
v_{k+1}=v_k+\frac{2\delta}{\uc{c:vdelay}(\uck{kfinaldelay})\sqrt{\ell_k}}\ge v_k+\frac{v_{\max}-v_k}{\sqrt{\ell_k}}.
\end{equation}
Note that, by \eqref{1star},
\begin{equation}
v_\infty\le v_+-\frac{\delta}{2\uc{c:vdelay}(\uck{kfinaldelay})}.
\end{equation}
By Corollay \ref{l:vdelay}, we already have that, for any $h\ge \uc{minhfinal}$,
\begin{equation}
p^h_{\uck{kfinaldelay}}\left(v_{\uck{kfinaldelay}}\right)\le e^{-\left(\ln L_{\uck{kfinaldelay}}\right)^{3/2}}.
\end{equation}
Now, assume that, some $k\ge \uck{kfinaldelay}$, we have that for any $h\ge \uc{minhfinal}$ and any $m\in M^h_{k}$
\begin{equation}
p^h_k(v_k)\le e^{-\left(\ln L_k\right)^{3/2}},
\end{equation}
then, applying Proposition \ref{metalemma} with $v_{\min}=v_k$ and using \eqref{defvkend}, we obtain
\begin{equation}
\begin{split}
\frac{p_{hL_{k+1}}\left(v_{k+1}\right)}{e^{-\left(\ln L_{k+1}\right)^{3/2}}}\le& 9\ell_k^4e^{\left(\ln L_{k+1}\right)^{3/2}}\left(p_{hL_k}(v_{\max})+\left(p_{hL_k}(v_{\min})\right)^2+\uc{c:decouple}e^{-\left(hL_k\right)^{1/4}}\right)\\
\le& 9\ell_k^4\left[ \uc{c:deviation}\left(\frac{\delta}{\uc{c:vdelay}(\uck{kfinaldelay})},\rho\right)  e^{(5/4)^{3/2}\left(\ln L_{k}\right)^{3/2}-2\left(\ln hL_k\right)^{3/2}} \right.\\
&\left.+e^{-\left(2-(5/4)^{3/2}\right)\left(\ln L_{k}\right)^{3/2}}+\uc{c:decouple}e^{-L_k^{1/4}+(5/4)^{3/2}\ln L_k   }\right],
\end{split}
\end{equation}
where we used Lemma \ref{l:deviation_interp} for the second inequality.
Using that $\left(\ln hL_k\right)^{3/2}\ge\left(\ln h\right)^{3/2}+\left(\ln L_k\right)^{3/2}$ and \eqref{3star}, we have
\begin{equation}
\begin{split}
\frac{p_{hL_{k+1}}\left(v_{k+1}\right)}{e^{-\left(\ln L_{k+1}\right)^{3/2}}}
\le& 9\ell_k^4\left[ 2e^{-\left(2-(5/4)^{3/2}\right)\left(\ln L_{k}\right)^{3/2}}+\uc{c:decouple}e^{-L_k^{1/4}+(5/4)^{3/2}\ln L_k   }\right]\\
\le&1,
\end{split}
\end{equation}
where we used \eqref{2star}. Thus \eqref{makesgoodforyou} holds for any $k\ge \uck{kfinaldelay}$, which concludes the proof of Proposition \ref{prop:final_delay}.
\end{proof}

%%%%%%%%%%%%%%%%%%%%%%%%%%%%%%%%%%%%%%%%%%%%%%%%%%%%%%%%%%%

%%%%%%%%%%%%%%%%%%%%%%%%%%%%%%%%%%%%%%%%%%%%%%%%%%%%%%%%%%%

\section{Strong decay below $v(\rho)>0$ or above $v(\rho)<0$}

\subsection{Sprinkling decoupling}

In this section, we state another decoupling result, similar to the lateral decoupling of Proposition \ref{p:decouple}. Let us briefly explain the difference between Proposition \ref{p:decouple} and the proposition below.\\
In Proposition  \ref{p:decouple}, we are taking advantage of the fact that, for the events we are interested in, we can always look at space-time boxes which are well-separated in space, compared to their time distance. This is because we were studying willing to control events of the form {\it ``the walker goes faster than $v>0$''}. This would ensure that, if we have enough bad boxes in a row, the first one and the last one will have a large space distance.\\
In the following, we will want to control events such that to bad space-time boxes could be ``on the top of each other'', which means that the two boxes would be at the same space position (but at different times). Hence, Proposition \ref{p:decouple} cannot be applied.\\
In the following {\it sprinkling decoupling}, we are able to deal with this case; nevertheless, the cost for this is that we need increase the density $\rho$ when we go from one scale to a bigger scale. In particular, on the contrary of Proposition \ref{p:decouple}, this technique cannot be used to prove precise statement about a fixed density $\rho$ (but rather about any density arbitrarily close to $\rho$).\\

Before stating the sprinkling decoupling, we need some notation. We use the same scales and boxes that we defined in Section \ref{s:scales}. Recall that we have now prove  Theorem \ref{l:v_+=v_-} and that we have define in \eqref{e:defv} the quantity $v(\rho)=v_+(\rho)=v_-(\rho)$ and recall the definition \eqref{e:defrho+-} of $\rho_+$.\\
From now on, we fix $\epsilon>0$. Note that, by definition of $\rho_+$,
\begin{equation}
v(\rho_++\epsilon)>0.
\end{equation}
\nck{k:minrhok}
Let us define $\uck{k:minrhok}(\epsilon)$ such that
\begin{equation}
\sum_{k\ge \uck{k:minrhok}} \frac{1}{L_k^{1/16}}\le\frac{\epsilon}{4}.
\end{equation}
Let us define
\begin{equation}\label{e:defrhok}
\rho_{\uck{k:minrhok}}=\rho+\frac{5}{4}\epsilon\text{, and }\rho_{k+1}=\rho_k+\frac{1}{L_k^{1/16}}\text{ for all }k\ge\uck{k:minrhok}.
\end{equation}
In particular, one can see that $\rho_\infty=\lim_{k\to\infty}\rho_k\le \rho+7\epsilon/4<\rho+2\epsilon$.\\
In the following, we say that a function $f$ of the environment is non-increasing is $f(\eta^\rho)$ stochastically dominates $f(\eta^{\rho+\epsilon})$. When we will apply the following result, the assumption $p_{\bullet}\ge p_\circ$ will matter (but, if this condition was violated, we would simply run the same argument by flipping the integer line).

\begin{proposition}\label{p:sprinkling}
\nc{c:distsprinkling}   \nc{c:sprinkling}
Consider the environment with law $\mathbf{P}^\rho_{EP}$ or $\mathbf{P}^\rho_{RW},$ with densities $\rho\in(0,1)$. Let $h\ge1$, $k\ge \uck{k:minrhok}$,  $w_1, w_2\in\mathbb{R}^2$ and recall the definition \eqref{e:B_L_h_w} of $B^h_{L_k}(w_{\cdot})$.
There exist constant $\uc{c:distsprinkling}$ and $\uc{c:sprinkling}=\uc{c:sprinkling}(\rho_+,\epsilon)$ such that, if
\begin{equation}
\pi_2(w_1)-\pi_2(w_2)\ge \uc{c:distsprinkling} hL_k,
 \end{equation}
 then, for any non-increasing functions $f_1$ and $f_2$, with values in $[0,1]$, supported respectively on $B^h_{L_k}(w_{1})$ and $B^h_{L_k}(w_{2})$,  we have that
   \begin{equation}
     \label{e:decouple2}
    \mathbb{E}^{\rho_{k+1}}\left[f_1f_2\right] \leq \mathbb{E}^{\rho_{k}}\left[f_1\right]\mathbb{E}^{\rho_{k}}\left[f_2\right]+ \uc{c:sprinkling} e^{-\frac{1}{\uc{c:sprinkling}}\left(hL_k\right)^{1/8}}.
   \end{equation}
\end{proposition}

\begin{proof}
When the law of the environment is $\mathbf{P}^\rho_{EP}$, this statement is a straightforward application of
Theorem 2.1.1 in \cite{rangel} and the case of non-increasing functions.\\
When the law of the environment is $\mathbf{P}^\rho_{RW}$, we need to provide some details in order to apply Theorem 3.4 of \cite{RWonRW2}. Note that, in our paper, we parametrized the Poisson cloud of random walks with a density $\rho\in(0,1)$. Nevertheless, at the initial time, each site contains a random number of particles which is distributed as Poisson of parameter $\lambda=-\ln(1-\rho)\in(0,\infty)$. In \cite{RWonRW2} they use densities on $(0,\infty)$, thus similar to what we denote here $\lambda$ (but denoted $\rho$ in \cite{RWonRW2}). Therefore, we should apply Corollary 3.1 of \cite{RWonRW} with $\rho=\lambda_k:=-\ln(1-\rho_k)$.
\nc{c:boundrhok}
If we let $\lambda_{k+1}=-\ln(1-\rho_{k+1})$, then by \eqref{e:defrhok}, their exists a constant $\uc{c:boundrhok}(\rho_+,\epsilon)$ such that
\begin{equation}
\lambda_{k+1}\ge\lambda_k\left(1+{\left(hL_k/\uc{c:boundrhok}\right)^{-1/16}}\right).
\end{equation}
As $f_1$ and $f_2$ are non-increasing, we have that
\begin{equation}
 \mathbb{E}^{\rho_{k+1}}\left[f_1f_2\right] \le  \mathbb{E}^{\lambda_k\left(1+{\left(hL_k/\uc{c:boundrhok}\right)^{-1/16}}\right)}\left[f_1f_2\right] \text{ and } \mathbb{E}^{\lambda_k\left(1+{\left(hL_k/\uc{c:boundrhok}\right)^{-1/16}}\right)}\left[f_1\right]\le\mathbb{E}^{\rho_{k}}\left[f_1\right].
\end{equation}
Now, a straightforward application of Theorem 3.4 of \cite{RWonRW2} provides the conclusion.
\end{proof}

%%%%%%%%%%%%%%%%%%%%%%%%%%%%%%%%%%%%%%%%%%%%%%%%%%%%%%%%%%%

\subsection{Ballisticity: proof of Proposition \ref{p:ballisticity}}

Again, we will only prove the first line of \eqref{e:ballisticity}, but the symmetric arguments holds for the second line.
As in Section \ref{s:decay_scale}, we need to define a sequence of velocities.
\nck{k:minkrho2}
First, fix $\uck{k:minkrho2}(\uc{c:boundrhok},\rho_+,\epsilon)\ge \uck{k:minrhok}$ such that
\begin{equation}
\sum_{k\ge \uck{k:minkrho2}}\frac{2(\uc{c:boundrhok}+1)}{\ell_k}\le \frac{v(\rho_++\epsilon)}{4}.
\end{equation}
Define
\begin{equation}
\tilde{v}_{\uck{k:minkrho2}}=\frac{7}{8}v(\rho_++\epsilon)\text{ and } \tilde{v}_{k+1}=v_k-\frac{2\uc{c:boundrhok}}{\ell_k}\text{ for all }k\ge\uck{k:minkrho2}.
\end{equation}
In particular, note that $\tilde{v}_{\uck{k:minkrho2}}<1$ and $\tilde{v}_\infty\ge \frac{5}{8}v(\rho_++\epsilon)$.
\begin{lemma}\label{l:determlemmasprink}
Let $h\ge1$, $k\ge \uck{k:minkrho2}$ and $m\in M^h_{k+1}$. If $\tilde{A}_m(v_{k+1})$ holds, then there exists $m_1=(h,k,w_1), m_2=(h,k,w_2)\in C_m$ such that $\pi_2(w_2)-\pi_2(w_1)\ge \uc{c:boundrhok} hL_k$ and both events $\tilde{A}_{m_1}(v_{k})$ and $\tilde{A}_{m_2}(v_{k})$ hold.
\end{lemma}
\begin{proof}
As $\tilde{A}_m(v_{k+1})$ holds, there exits $y\in[0,hL_{k+1})\times\{0\}$ such that $X^y_{hL_{k+1}}-\pi_1(y)\le v_{k+1}hL_k$.\\
Recall that the set $C_m$ is such that, for any $i\in \{0,\dots, \ell_{k}-1\}$ there exists a unique $m'\in C_m$ such that $X^y_{ihL_k}\in I_{m'}$. Denote $M(X^y)$ these (random) indices. In order to conclude, it is enough to prove that there exists at least $\uc{c:boundrhok}+1$ indices $m'\in M(X^y)$ such that $\tilde{A}_{m'}(v_k)$ happens.\\
Assume that there is at most $\uc{c:boundrhok}$ indices $m'\in M(X^y)$ such that $\tilde{A}_{m'}(v_k)$ happens and let us find a contradiction. Note that, if $m'\in M(X^y)$, then $X^y_{ihL_k}\in I_{m'}$ for some $i\in \{0,\dots, \ell_{k}-1\}$ and
\begin{itemize}
\item if $m'\in\tilde{A}_{m'}(v_k)$, then  $X^y_{(i+1)hL_k}-X^y_{ihL_k}\ge-hL_k$ by \eqref{e:lipschitz};
\item if $m'\in\tilde{A}^{\text{c}}_{m'}(v_k)$, then  $X^y_{(i+1)hL_k}-X^y_{ihL_k}\ge-v_khL_k$.
\end{itemize}
Therefore, we have that
\begin{equation}
\begin{split}
X^y_{hL_{k+1}}-\pi_1(y)&\ge -\uc{c:boundrhok}hL_k+\left(\ell_k-\uc{c:boundrhok}\right)v_khL_k\\
&\ge hL_{k+1}\left(v_-\frac{\uc{c:boundrhok}(1+v_k)}{\ell_k}\right)\\
&>v_{k+1}hL_{k+1},
\end{split}
\end{equation}
which contradict that $X^y_{hL_{k+1}}-\pi_1(y)\le v_{k+1}hL_k$.
\end{proof}
\nck{k:lastone}
Let $\uck{k:lastone}(\uc{c:boundrhok},\rho_+,\epsilon)\ge\uck{k:minkrho2}$ be such that, for any $k\ge\uck{k:lastone}$, we have
\begin{equation}\label{e:lastone}
3\ell_k^2\left(e^{-4(2-(5/4)^{3/2})\left(\ln L_k\right)^{3/2}}  +\uc{c:sprinkling} e^{-\frac{1}{\uc{c:sprinkling}}\left(hL_k\right)^{1/8}+8\left(\ln L_k\right)^{3/2}}\right)\le 1.
\end{equation}
\begin{lemma}\label{l:lemmasprink}
If, for some  $k\ge \uck{k:lastone}$ and $h\ge1$, we have that
\begin{equation}
\tilde{p}_{hL_{k}}(v_k,\rho_k)\le e^{-4\left(\ln L_k\right)^{3/2}},
\end{equation}
then we have that
\begin{equation}
\tilde{p}_{hL_{k+1}}(v_{k+1},\rho_{k+1})\le e^{-4\left(\ln L_{k+1}\right)^{3/2}}.
\end{equation}
\end{lemma}
\begin{proof}
Similarly as what we did in the proof of Proposition \ref{metalemma}, but using Lemma \ref{l:determlemmasprink} and Proposition \ref{p:sprinkling} for the decoupling, we have that
\begin{equation}
\begin{split}
\frac{\tilde{p}_{hL_{k+1},\rho_{k+1}}(v_{k+1})}{e^{-4\left(\ln L_{k+1}\right)^{3/2}}} \le& 3\ell_k^2e^{4\left(\ln L_{k+1}\right)^{3/2}}\left((\tilde{p}_{hL_{k}}(v_k,\rho_k))^2  +\uc{c:sprinkling} e^{-\frac{1}{\uc{c:sprinkling}}\left(hL_k\right)^{1/8}}\right)\\
\le&3\ell_k^2\left(e^{-4(2-(5/4)^{3/2})\left(\ln L_k\right)^{3/2}}  +\uc{c:sprinkling} e^{-\frac{1}{\uc{c:sprinkling}}\left(hL_k\right)^{1/8}+8\left(\ln L_k\right)^{3/2}}\right)\\
\le&1,
\end{split}
\end{equation}
where we used \eqref{e:lastone}. In order to apply Proposition \ref{p:sprinkling}, we used that the event $\tilde{A}_m(v)$ are non-increasing events (recall that we assume $p_{\bullet}\ge p_\circ$).
This concludes the proof.
\end{proof}

\nc{c:vpos}
\begin{proposition}\label{p:sprink}
There exists a constant $\uc{c:vpos}\ge1$ such that, for all $k\ge\uck{k:lastone}$, we have that, for any $\rho\ge\rho_++2\epsilon$,
\begin{equation}
\tilde{p}_{\uc{c:vpos}L_k}\left(\frac{5}{8}v(\rho_++\epsilon),\rho\right)\le e^{-4\left(\ln L_k\right)^{3/2}}.
\end{equation}
\end{proposition}
\begin{proof}
Note that $v_{\uck{k:lastone}}<v(\rho_++\epsilon)=v_-(\rho_++\epsilon)$, thus there exists $\uc{c:vpos}\ge1$ such that
\begin{equation}
\tilde{p}_{\uc{c:vpos}L_{\uck{k:lastone}}}(v_{\uck{k:lastone}},\rho_{\uck{k:lastone}})\le e^{-4\left(\ln L_{\uck{k:lastone}}\right)^{3/2}}.
\end{equation}
Thus, by Lemma \ref{l:lemmasprink}, we have that, for all $k\ge {\uck{k:lastone}}$,
\begin{equation}
\tilde{p}_{\uc{c:vpos}L_{k}}(v_{k},\rho_{k})\le e^{-4\left(\ln L_{k}\right)^{3/2}}.
\end{equation}
Finally, noting that the events $\tilde{A}_m(v)$ are non-decreasing in $v$ and non-increasing in $\rho$, we can conclude by noting that, for $k\ge {\uck{k:lastone}}$, $v_k\ge \frac{5}{8}v(\rho_++\epsilon)$ and $\rho_k\le \rho_++2\epsilon$.
\end{proof}

\begin{proof}[Proof of Proposition \ref{p:ballisticity}]
From Proposition \ref{p:sprink}, we simply need an interpolation argument, very similar to the one exposed in Section \ref{s:proof_deviation_interp}. Therefore, we will note entirely re-do it but simply point out the differences.\\
First, one should define $v=v(\rho_++\epsilon)/2$ and $v'=\frac{5}{8}v(\rho_++\epsilon)>v$. Second, one should consider events of the form $\tilde{A}_{m'}(v')$ (instead of ${A}_{m'}(v')$). In particular, this would yield the opposite inequality in \eqref{e:inter1}. The inequality \eqref{e:inter2} becomes
\begin{equation}
H-{\left\lfloor \frac{H}{\uc{c:h_hat}  L_{\bar{k}}} \right\rfloor \uc{c:h_hat} L_{\bar{k}}}\ge- \frac{v'-v}{2}H,
\end{equation}
and thus we obtain the opposite inequality in \eqref{e:inter3}. The rest of the proof is identical.
\end{proof}

%%%%%%%%%%%%%%%%%%%%%%%%%%%%%%%%%%%%%%%%%%%%%%%%%%%%%%%%%%%
%%%%%%%%%%%%%%%%%%%%%%%%%%%%%%%%%%%%%%%%%%%%%%%%%%%%%%%%%%%

\section{Proof of Theorem \ref{t:speed}} \label{billburr}

%%%%%%%%%%%%%%%%%%%%%%%%%%%%%%%%%%%%%%%%%%%%%%%%%%%%%%%%%%%
%%%%%%%%%%%%%%%%%%%%%%%%%%%%%%%%%%%%%%%%%%%%%%%%%%%%%%%%%%%

In order to conclude the proof of Theorem \ref{t:speed}, we simply need to use, as inputs, results from \cite{RWonRW} for PCRW and from \cite{francoiss} for SSEP. Theorem \ref{t:speed} is given by  the following result. These results are based on regeneration structures.

\begin{proposition}
Consider the environment $\eta^\rho$ with density $\rho\in(0,1)$ under either the measure $\mathbf{P}^\rho_{EP}$ or $\mathbf{P}^\rho_{RW}$. Recall the definition \eqref{e:defv} of $v(\cdot)$ and the definition \eqref{e:defrho+-} of $\rho_-$ and $\rho_+$.\\
First, for any $\rho\in(0,1)\setminus \{\rho_-,\rho_+\}$,
\begin{equation}\label{creepy}
 \frac{X_n}{n} \to v(\rho), \qquad \mathbb{P}^{\rho}-\text{almost surely.}
\end{equation}
Second, we have, for  any $\rho\in(0,\rho_-)\cup(\rho_+,1)$, a functional central limit theorem for $X_n$ under $\mathbb{P}^\rho$, that is
\begin{equation}\label{supercreepy}
\left(\frac{X_{\lfloor nt\rfloor}-ntv(\rho)}{\sqrt{n}}\right)_{t\ge0} \stackrel{(d)}{\to} \left(B_t\right)_{t\ge0},
\end{equation}
where $\left(B_t\right)_{t\ge0}$ is a non-degenerate Brownian motion and where the convergence in law holds in the Skorohod topology.
\end{proposition}
\begin{proof}
First, if $\rho\in(\rho_-,\rho_+)$, then \eqref{creepy} follows  from Theorem \ref{t:speed_zero}. Second, we prove \eqref{creepy} and \eqref{supercreepy} for $\rho\in(\rho_+,1)$ and the conclusion will follow by symmetry.\\

From now on, we suppose that $\rho_+<1$ and fix $\rho\in(\rho_+,1)$. Note that, as $\rho_+<1$, we have that $p_\bullet>1/2$. Indeed, as $p_\circ\le p_\bullet$, if $p_\bullet\le 1/2$, then the random walker would always be on the left of a simple random walk (regardless of the value of $\rho$), therefore we could not have that $\rho_+<1$.\\

In the following, we apply existing results in order to prove that the speed of the walk exists. The fact that the speed coincides with $v(\rho)$ is a direct consequence of Theorem \ref{l:v_+=v_-} and of the definition of $v_+$ and $v_-$.\\

\noindent
{\bf Case I: under $\mathbf{P}^\rho_{RW}$.}\\
Here, we want to apply Theorem 1.4 from \cite{RWonRW} and we thus simply need to check that the conditions of the statement are satisfied. Recall that we assumed $\rho_+<1$, $\rho\in(\rho_+,1)$ as well as $p_\bullet>1/2$ and $p_\circ>0$. Using the notation of Theorem 1.4 from \cite{RWonRW}, $p_\bullet>1/2$ corresponds to $v_\bullet>0$ and $p_\circ>0$ corresponds to $v_\circ>-1$.\\
Let us define $\epsilon=\rho-\rho_+>0$. Let us define $v_{\star}=v(\rho_++\epsilon)/2$. By Proposition \ref{p:ballisticity}, we have that $v_\star>0$ and there exists a constant $\uc{c:ball+}(\epsilon)$  such that
\begin{equation}
\tilde{p}_H (v(\rho_++\epsilon)/2, \rho)\leq \uc{c:ball+}\exp\big(-2\ln^{3/2}H \big).
\end{equation}
\nc{c:ball++}
This implies that, for any $L\in \mathbb{N}$, there exists a constant $\uc{c:ball++}\in(0,\infty)$ such that
\begin{equation}
\begin{split}
\mathbb{P}^\rho\left[\exists n\in\mathbb{N}:\ X_n\le nv_\star -L\right]&\le \sum_{k=\lfloor L/2\rfloor}^\infty \uc{c:ball+}\exp\big(-2\ln^{3/2}k \big)\\
&\le  \frac{1}{\uc{c:ball++}}\exp\big(-\uc{c:ball++}\ln^{3/2}L \big).
\end{split}
\end{equation}
By Theorem  1.4 from \cite{RWonRW}, the conclusions $(a)$ and $(b)$ of Theorem 1.2 of \cite{RWonRW} hold, and these are precisely \eqref{creepy} and \eqref{supercreepy}.\\

\noindent
{\bf Case II: under $\mathbf{P}^\rho_{EP}$.}\\
In order to prove this case, we can follow line by line the proof of Theorem 1.1 in \cite{francoiss}, Section 5. 
Note that we will not need to suppose the same assumptions as in Theorem 1.1 in \cite{francoiss} (in particular we do not need to take the rate $\gamma$ to be large).
Ineed, these assumptions are used in the proof of Proposition 3 in \cite{francoiss}. But, once we assume the conclusions of Proposition 3 in \cite{francoiss}, then we can follow line by line the proof of Theorem 1.1 in Section 5 of \cite{francoiss}, without any assumption.\\
The conclusion of Proposition 3.3 in \cite{francoiss} is almost provided to us by Proposition \ref{p:ballisticity} as \eqref{e:ballisticity} yields
\begin{equation}\label{flexao}
 \tilde{p}_H (v(\rho_++\epsilon)/2, \rho)\leq \uc{c:ball+}\exp\big(-2\ln^{3/2}H \big).
 \end{equation}
The difference with the conclusion of Proposition 3.3 in \cite{francoiss} is that we obtain a super-polynomial decay instead of a stretched exponential decay. Nevertheless, a close inspection of the proof in \cite{francoiss} shows that, along the proof, $\tilde{p}_H (v(\rho_++\epsilon)/2, \rho)$ is always simply bounded by some polynomial decay. To ease the work of the reader, we list here the places in Section 5 of \cite{francoiss} where Proposition 3.3 of \cite{francoiss} is applied:
\begin{itemize}
\item[-] page 31: Proposition 3.3 is used to prove Proposition 5.1 of \cite{francoiss}, in the paragraph just before the last paragraph on p.~31. Even though the first inequality is not true anymore, we still have that $P_0(X_n\le\frac{v}{2}n)\le 1/n^{q+2}$, keeping the notation therein.
\item[-] page 35: in the proof of $(5.24)$ of \cite{francoiss}, it is used that $\liminf X_n/n>v$ almost surely, which is an easy consequence of \eqref{flexao} and Borel-Cantelli Lemma.
\item[-] page 39: Proposition 3.3 is used to prove Proposition 5.6 of \cite{francoiss} and, even if the second inequality is not true anymore, we can directly write the third inequality by \eqref{flexao}.
\end{itemize}
Once these minor modifications are made, we can apply the whole proof from Section 5 of \cite{francoiss} and prove  \eqref{creepy} and \eqref{supercreepy}.
\end{proof}

\noindent
{\bf Acknowledgement:} DK and MH warmly thanks NYU Shanghai, IMPA and MFO-Oberwolfach where this work was partially done. MH was partially supported by CNPq grant 406659/2016-8.

\bibliographystyle{alpha}
\bibliography{cavalinhos_2}

\end{document}